\documentclass[11pt,reqno]{amsart}

\usepackage{amsfonts,amssymb,amsmath}

\usepackage[cp1251]{inputenc}
\usepackage[T2B]{fontenc}
\usepackage[english]{babel}
\usepackage[mathscr]{eucal}
\usepackage{calrsfs}

\usepackage{tikz}
\usepackage{ragged2e}

\usepackage{tikz}  
\usepackage[all]{xy}

\newtheorem{theorem}{Theorem}[section]

\newtheorem{corollary}[theorem]{Corollary}
\newtheorem{definition}[theorem]{Definition}

\theoremstyle{remark}
\newtheorem{remark}[theorem]{Remark}
\newtheorem{example}[theorem]{Example}

\def\loc{\text{\rm loc}}

\usepackage[left=2.0cm,right=2.2cm,
    top=2.5cm,bottom=2.8cm]{geometry}

\usepackage{graphicx}
\usepackage{float}
\usepackage{subcaption} 

\begin{document}

\centerline{\bf NORM RELATED INEQUALITIES FOR FRACTIONAL INTEGRALS}

\smallskip\begin{center}
{\bf Elena P. Ushakova$~^{1,2,3}$}\footnote{Corresponding author. E--mail address: elenau@inbox.ru.\\ The research work of the first author related to Section 5 
was performed at Steklov Mathematical Institute of Russian Academy of Sciences under financial support of the Russian Science Foundation (project 19-11-00087). 
The rest part of the paper was carried out within the framework of the State Tasks of Ministry of Education and Science of Russian Federation for V.A. Trapeznikov 
Institute of Control Sciences of Russian Academy of Sciences, Computing Center of Far--Eastern Branch of Russian Academy of Sciences and ITMO University. The both authors were also 
partially supported by the Russian Foundation for Basic Research (project 19--01--00223).} {\bf and Kristina E. Ushakova$~^{4}$}\end{center}

\smallskip
\centerline{\textit {$~^1$ V.A. Trapeznikov Institute of Control Sciences of RAS, Moscow, Russia}}
\centerline{\textit {$~^2$ Steklov Mathematical Institute of RAS, Moscow, Russia} }
\centerline{\textit {$~^3$ Computing Center of FEB RAS, Khabarovsk, Russia}} 
\centerline{\textit {$~^4$ ITMO University, Saint--Petersburg, Russia}}

\medskip



{\small {\bf Abstract.} Fractional spline wavelet systems are considered in the work. Molecular structure of elements of such systems admits estimates 
connecting norms of fractional integrals' images and pre--images in Besov spaces.
}

\smallskip
\section{Introduction}
For $f\in L_1^\mathrm{loc}(\mathbb{R})$ we consider the left-- and the right--hand side Riemann--Liouville operators
\begin{equation}\label{left}
I_{a+}^{\boldsymbol{\alpha}} f({x}):=\frac{1}{\Gamma(\boldsymbol{\alpha})}\int_{a}^{x}{(x-y)^{\boldsymbol{\alpha}-1}}{f({y})}\,dy\qquad(x>a)
\end{equation} and
\begin{equation}\label{right}
I_{a-}^{\boldsymbol{\alpha}} f({x}):=\frac{1}{\Gamma(\boldsymbol{\alpha})}\int_{x}^{a}{(y-x)^{\boldsymbol{\alpha}-1}}{f({y})}\,dy\qquad(x<a)
\end{equation} of positive fractional order $\boldsymbol{\alpha}$ \cite{SKM} in Besov spaces $B_{pq}^{s}(\mathbb{R},w)$ with 
$1\le p<\infty$, $0<q\le\infty$, $s\in\mathbb{R}$ and Muckenhoupt weights $w$ 
(see \S\,\ref{BS} for definitions).

The main purpose of this paper is the study of relations between norms of images and pre--images of operators \eqref{left} and \eqref{right} in 
$B_{pq}^{s}(\mathbb{R},w)$. For simplicity, we assume $f(x)\equiv 0$ for $x\in(-\infty,a)$ or $x\in(a,+\infty)$. 

Connections between images and pre--images of operators in $B_{pq}^s(\mathbb{R},w)$ 
have been studied in \cite[Theorem 2.3.8]{Tr1}, \cite{N}, \cite[Theorem 2.20]{R}, \cite[\S\,4]{IS}, \cite[p.\,23]{Tr6}. Norm related inequalities 
in $B_{pq}^{s}(\mathbb{R},w)$ for integrals \eqref{left} and \eqref{right} of natural orders $\boldsymbol{\alpha}$ were considered in \cite{UshIm}. 
This work continues the study 
of the same problem by extending the results obtained in \cite[Theorems 5.1, 5.2]{UshIm} to fractional $\boldsymbol{\alpha}>0$.

Our instruments are decompositions from \cite[Theorem 11.4]{Rou} and \cite[\S\,4.3]{PSI} (see also \cite[Theorem 4.7]{RMC})
of molecular and atomic types, respectively (see \S\,\ref{DEC}). We use them by applying Riesz bases generated by spline wavelet systems of fractional and natural orders.   
In \S\,\ref{SP} explicit formulae are given for elements of such systems since we need them for establishing our main results in \S\,\ref{MR}. 
More precisely, in Theorem \ref{ImagesA}, we obtain conditions for the validity of embedding inequalities for (sub)spaces of images and pre--images of 
the operators $I_{a^\pm}^{\boldsymbol{\alpha}}$ with fractional
$\boldsymbol{\alpha}>0$ in $B_{pq}^{s}(\mathbb{R},w)$. In particular, for 
$I_{0^\pm}^{\boldsymbol{\alpha}}$ one of our results reads 
\begin{theorem}\label{T11} Let $p>1$, $q>0$ and $s\in\mathbb{R}$. Suppose $u,w$ are Muckenhoupt weights and $f\in L_1^\textrm{loc}(\mathbb{R})$. 
For fractional $\boldsymbol{\alpha}>0$ let $I_{0^\pm}^{\boldsymbol{\alpha}}$ be defined by \eqref{left} and \eqref{right}. 
Suppose $f\equiv 0$ on $\mathbb{R}_\mp$ for $I_{0^\pm}^{\boldsymbol{\alpha}}f$, where $\mathbb{R}_-:=(-\infty,0)$ and $\mathbb{R}_+:=(0,\infty)$. {\rm (i)} 
Assume that the following equivalences hold for the both 
$v=u$ and $v=w$: \begin{equation}\label{usl}\int_{2^{-d}(r-1/2)}^{2^{-d}(r+1/2)}v\approx 2^{-d}v(2^{-d}r)\qquad \textrm{for all } d\in\mathbb{N}_0
:=\{0\}\cup\mathbb{N}\ \textrm{and } r\in\mathbb{Z}.\end{equation} Then
$I_{0^\pm}^{\boldsymbol{\alpha}}f\in {B}_{pq}^{s}(\mathbb{R},w)$ if $f\in {B}_{pq}^{s+\boldsymbol{\alpha}}(\mathbb{R},u)$ provided 
$\mathscr{N}_{0^\pm}^{\boldsymbol{\alpha}}<\infty$, where
\begin{align*}\mathscr{N}_{0^+}^{\boldsymbol{\alpha}}=&\sup_{\tau\in \mathbb{Z}^+}
\biggl(\sum_{r\ge\tau}(r-\tau+1)^{p(2\boldsymbol{\alpha}-1)}w(r)\biggr)^{{1}/{p}}\biggl(\sum_{0\le r\le\tau} 
{u}(r)^{1-p'}\biggr)^{{1}/{p'}}\\&+\sup_{\tau\in \mathbb{Z}^+}
\biggl(\sum_{r\ge\tau}w(r)\biggr)^{{1}/{p}}\biggl(\sum_{0\le r\le\tau} (\tau-r+1)^{p'(2\boldsymbol{\alpha}-1)}
{u}(r)^{1-p'}\biggr)^{{1}/{p'}},\\
\mathscr{N}_{0^-}^{\boldsymbol{\alpha}}=&\sup_{\tau\in \mathbb{Z}^-}
\biggl(\sum_{0\ge r\ge\tau}(r-\tau+1)^{p(2\boldsymbol{\alpha}-1)}w(r)\biggr)^{{1}/{p}}\biggl(\sum_{r\le\tau} 
{u}(r)^{1-p'}\biggr)^{{1}/{p'}}\\&+\sup_{\tau\in \mathbb{Z}^-}
\biggl(\sum_{0\ge r\ge\tau}w(r)\biggr)^{{1}/{p}}\biggl(\sum_{r\le\tau} (\tau-r+1)^{p'(2\boldsymbol{\alpha}-1)}
{u}(r)^{1-p'}\biggr)^{{1}/{p'}}
\end{align*} and $\mathbb{Z}^+:=\mathbb{Z}\cap[0,\infty)$, $\mathbb{Z}^-:=\mathbb{Z}\cap(-\infty,0]$.
Moreover, \begin{equation*}\|I_{0^\pm}^{\boldsymbol{\alpha}}f\|_{{B}_{pq}^{s}(\mathbb{R},w)}
\lesssim \mathscr{N}_{0^\pm}^{\boldsymbol{\alpha}}\|f\|_{{B}_{pq}^{s+\boldsymbol{\alpha}}(\mathbb{R},u)}.\end{equation*}
{\rm (ii)} If $I_{0^\pm}^{\boldsymbol{\alpha}}f\in {B}_{pq}^{s}(\mathbb{R},w)$ then 
$f\in {B}_{pq}^{s-{\boldsymbol{\alpha}}}(\mathbb{R},w)$, provided $r_w<\boldsymbol{\alpha}$ {\rm(}see \eqref{param}{\rm)}, besides,
$$\|f\|_{{B}_{pq}^{s-\boldsymbol{\alpha}}(\mathbb{R},w)}\lesssim\|I_{0^\pm}^{\boldsymbol{\alpha}}f\|_{{B}_{pq}^{s}(\mathbb{R},w)}.$$ 
For $\boldsymbol{\alpha}\in(0,1)$ the assertion {\rm(ii)} is unconditionally true in the case $w\equiv 1$.
\end{theorem}
 
Throughout the paper relations of the type $A\lesssim B$ mean that $A\le cB$ with
some constant $0<c<\infty$ depending, possibly, on number parameters. We write
$A\approx B$ instead of $A\lesssim B \lesssim A$ and $A\simeq B$ instead of $A=cB$. We stand $\mathbb Z$, $\mathbb N$ and $\mathbb R$ for integer, natural and 
real numbers, respectively. 
By $\mathbb{N}_0$ we denote the set $\mathbb{N}\cup\{0\}$. 
The symbol 
$\Gamma(\cdot)$ stands for the Gamma function, 
 $[s]$ --- for the greatest integer less than or equal to $s\in\mathbb{R}$. 
We put $r':=
{r}/({r-1})$ if $0<r<\infty$ and $r'=1$ for $r=\infty$. 
We say that $f\in L_1^\mathrm{loc}(\mathbb{R})$ if $f\in L_1(\Omega)$ for every compact subset $\Omega$ of $\mathbb{R}$. Marks $:=$ and
$=:$ are used for introducing new quantities. We abbreviate $h(\Omega):=\int_\Omega h(x)\, dx$,
where {$\Omega \subset\mathbb{R}$} is some bounded, measurable set.

\section{Besov spaces with Muckenhoupt weights}\label{BS}
Let a function $w$ be locally integrable and almost everywhere positive on $\mathbb{R}$ (a weight). By $L^r(\mathbb{R})$, $0<r\le\infty$, we denote the Lebesgue space of all measurable functions $f$ on $\mathbb{R}$ quasi--normed by $$\|f\|_{L^r(\mathbb{R})}:=\biggl(\int_{\mathbb{R}}|f(x)|^r\,dx\biggr)^{{1}/{r}}$$ with the usual modification if $r=\infty$.

\begin{definition}{\rm (\cite[Chapter~V]{S}) A weight $w$ belongs to the \textit{Muckenhoupt class} $\mathscr{A}_p$, $1<p<\infty$, if \begin{equation*}
\sup\frac{w(B)}{|B|}
\biggl(\frac{1}{|B|}\int_B w^{1-p'}\biggr)^{\frac{p}{p'}}<\infty,\end{equation*} where supremum is taken over all balls $B\subset\mathbb{R}$;\ $w\in\mathscr{A}_1$ if supremum over all balls $B\subset\mathbb{R}$ of the form
 \begin{equation*}
\sup \frac{w(B)}{|B|}\,\|1/w\|_{L^\infty(B)}<\infty\end{equation*} is finite;
\textit{Muckenhoupt class} $\mathscr{A}_\infty$ is given by $\mathscr{A}_\infty=\bigcup_{p\ge 1}\mathscr{A}_p.$
}\end{definition}

We refer to \cite[Chapter~V]{S} and e.g. \cite[Lemma 1.3]{HSc}, \cite[Lemma 2.3]{HP} for properties of the $\mathscr{A}_\infty$ class. One of them is the doubling property: there exists a constant $c>0$ such that for any $\delta>0$ and $z\in\mathbb{R}$
\begin{equation}\label{doubl}w\bigl(B_{2\delta}(z)\bigr)\le cw\bigl(B_{\delta}(z)\bigr)\end{equation} holds for all arbitrary balls $B_{\delta}(z):=\{x\in\mathbb{R}\colon |z-x|<\delta\}$ and $B_{2\delta}(z)$. If $2^r$ is the smallest constant $c$ for which \eqref{doubl} holds, then $r$ is called the doubling exponent of $w$. 
The same meaning has another characteristic parameter of the Muckenhoupt class, that is the number 
\begin{equation}\label{param}r_w:=\inf\{r\ge 1\colon w\in\mathscr{A}_r\}<\infty.\end{equation} In our proofs we shall use the following property of 
$w\in\mathscr{A}_p$, $p\ge 1$, which holds with some $c=c(w)>0$:
\begin{equation}\label{Muck}
\frac{|E|}{|B|}\le c \Bigl(\frac{w(E)}{w(B)}\Bigr)^{1/p} \qquad (E\subset B).\end{equation}  

Examples of weights $w$ belonging to $\mathscr{A}_p$, $1\le p<\infty$, can be {seen} in e.g. \cite[Examples 1.5]{HSc} and 
\cite[Remark 2.4, Example 2.7]{HP}. Alternative definitions, further properties and examples of Muckenhoupt weights can be found in 
\cite{S} and \cite[Lemma 1.4]{HSc} (see also \cite{HP,M,M1,M2,Tr1}).

Let $1\le p<\infty$, $0<q\le\infty$ and $s\in\mathbb{R}$. To define Besov spaces $B_{pq}^{s}(\mathbb{R},w)$ we introduce the Schwartz space 
$\mathscr{S}(\mathbb{R})$ of all complex--valued rapidly decreasing, infinitely differentiable functions on $\mathbb{R}$. By $\mathscr{S}'(\mathbb{R})$ we denote its topological dual, the space of tempered distributions on $\mathbb{R}$. Defining the Fourier transform 
\begin{equation*} 
\widehat{\varphi}(\xi)=(2\pi)^{-1/2}\int_{\mathbb{R}}\mathrm{e}^{-i\xi x}\varphi(x)\,dx \qquad (\xi\in\mathbb{R})
\end{equation*}
for $\varphi\in\mathscr{S}(\mathbb{R})$, we fix such a $\varphi$ having $\textrm{supp}\,\hat{\varphi}\subseteq\bigl\{\xi\in\mathbb{R}\colon 1/2\le|\xi|\le 2\bigr\}$ and satisfying $|\hat{\varphi}(\xi)|\ge c>0$ for $3/5\le|\xi|\le 5/3$, and set $\varphi_\nu(x):=2^{\nu-1}\varphi(2^{\nu-1} x)$ for $\nu\in\mathbb{N}$. In addition, we choose ${\varphi_0}\in\mathscr{S}(\mathbb{R})$ with $\textrm{supp}\,\hat{{\varphi_0}}\subseteq\bigl\{\xi\in\mathbb{R}\colon |\xi|\le 2\bigr\}$ satisfying $|\hat{{\varphi_0}}(\xi)|\le c>0$ for $|\xi|\le 5/3$.
The Besov spaces $B_{pq}^{s}(\mathbb{R},w)$ with Muckenhoupt weight $w$ \cite[Definition 11.1]{Rou} is the collection of all distributions $f \in \mathscr{S}'(\mathbb{R})$ such that
\begin{equation}\label{BspqS}
  \|f\|_{B_{pq}^{s}(\mathbb{R},w)}: =\bigl\|{\varphi_0}\ast f\bigr\|_{L^p(\mathbb{R},w)}+\biggl\|\sum_{\nu\in\mathbb{N}}2^{q\nu s}\bigl\|\varphi_\nu\ast f\bigr\|_{L^p(\mathbb{R},w)}^q\biggr\|^{1/q} <\infty.
\end{equation} Here, for $r\ge 1$ given and $w$ fixed, symbol $L^r(\mathbb{R},w)$ stands for the weighted Lebesgue space normed by $\|f\|_{L^r(\mathbb{R},w)}:=\|w^{1/r}f\|_{L^r(\mathbb{R})}$ with the usual modification for $r=\infty$. 
The \eqref{BspqS} admits the usual modification if $q=\infty$. Definition of $B_{pq}^{s}(\mathbb{R},w)$ is independent of the choice of $\varphi$ and ${\varphi}_0$.

For $\tau\in\mathbb{Z}$ and $\nu\in\mathbb{N}_0$ we introduce dyadic segments $Q_{\nu\tau}:=\Bigl[\frac{\tau}{2^\nu},\frac{\tau+1}{2^\nu}\Bigr]$ with the lower left corner $x_{Q_{\nu\tau}}=2^{-\nu}\tau$. In what follows we shall need spaces ${b}_{pq}^{s}(w)$ consisting of all sequences ${\lambda}=\{{\lambda}_{\nu\tau}\}$ such that
\begin{equation*}
  \|{\lambda}\|_{{b}_{pq}^{s}(w)}: =\biggl\|\sum_{\nu\in\mathbb{N}_0}2^{q\nu s}\Bigl\|\sum_{\tau\in\mathbb{Z}}|{\lambda}_{\nu\tau}|\chi_{Q_{\nu\tau}}\Bigr\|_{L^p(\mathbb{R},w)}^q\biggr\|^{1/q} <\infty.
\end{equation*}

Our results are partially based on ${B}_{pq}^{s}(\mathbb{R},w)$ and ${b}_{pq}^{s}(w)$ norm related inequalities using molecular representation of elements 
from $B_{pq}^{s}(\mathbb{R},w)$ \cite{Rou}. 

Let $D^\gamma$, $\gamma\in\mathbb{N}_0$, stand for derivatives. 
Assume $0<\delta\le 1$, $M>0$ and $N\in\mathbb{Z}$. A function $m_Q$ is called a smooth $(\delta,M,N)-$ molecule for $Q=Q_{\nu\tau}$ dyadic with $\nu\in\mathbb{N}$ if
\begin{itemize}
\item[(M1)] $\int_\mathbb{R}x^\gamma m_Q(x)\,dx=0$~ for~ $0\le\gamma\le N$,
\item[(M2)] $|m_Q(x)|\le 2^{\nu/2}\Bigl(1+{2^{\nu}}{|x-x_Q|}\Bigr)^{-\max\{M,M-s\}}$,
\item[(M3)] $|D^\gamma m_Q(x)|\le 2^{\nu/2+\nu\gamma}\Bigl(1+{2^{\nu}}{|x-x_Q|}\Bigr)^{-M}$~ if~ $0\le\gamma\le [s]$,
\item[(M4)] $|D^\gamma m_Q(x)-D^\gamma m_Q(y)|\le 2^{\nu/2+\nu\gamma
+\nu\delta}|x-y|^{\delta}\sup_{|z|\le|x-y|}
\Bigl(1+{2^{\nu}}{|x-z-x_Q|}\Bigr)^{-M}$~ if~ $0\le\gamma=[s]$.
\end{itemize} 
It is understood that (M1) is void if $N<0$, and (M3)--(M4) are void if $s<0$.

For $Q=Q_{0\tau}$, that is $Q_{\nu\tau}$ with $\nu=0$, a function $m_Q$ is a smooth $(\delta,M,N)-$ molecule if it satisfies
\begin{itemize}
\item[(M2*)] $|m_Q(x)|\le \bigl(1+{|x-x_Q|}\bigr)^{-M}$,
\item[(M3*)] $|D^\gamma m_Q(x)|\le \Bigl(1+{|x-x_Q|}\Bigr)^{-M}$~ if~ $0<\gamma\le [s]$,
\item[(M4*)] $|D^\gamma m_Q(x)-D^\gamma m_Q(y)|\le |x-y|^{\delta}\sup_{|z|\le|x-y|}
\Bigl(1+{|x-z-x_Q|}\Bigr)^{-M}$~ if~ $0\le\gamma=[s]$.
\end{itemize}
As before, (M3*)--(M4*) are void if $s<0$.

We say $\{m_Q\}=\{m_{Q_{\nu\tau}}\}$, with $\nu\in\mathbb{N}_0$ and $\tau\in\mathbb{Z}$, is a family of smooth molecules for ${B}_{pq}^{s}(\mathbb{R},w)$ if each $m_Q$ is a $(\delta,M,N)-$ molecule and for $\nu\in\mathbb{N}$ 
\begin{itemize} \item[(M.i)] $s-[s]<\delta\le 1$,
\item[(M.ii)] $M>J$, where $J=r_w/p+1/p'$ if $p>1$ and $J=r_w$ if $p=1$,
\item[(M.iii)] $N=\max\bigl\{[J-s-1],-1\bigr\}$. \end{itemize} 

Connection between ${B}_{pq}^{s}(\mathbb{R},w)$ and ${b}_{pq}^{s}(w)$ norms in one direction is regulated by the following 
\begin{theorem}\label{T0} \cite[Theorem 11.4]{Rou}
Let $s\in\mathbb{R}$, $1\le p<\infty$, $0<q\le\infty$ and $w\in\mathscr{A}_\infty$. Suppose $\{m_Q\}$ is a family of smooth molecules for ${B}_{pq}^{s}(\mathbb{R},w)$ 
and $f\in\mathscr{S}'(\mathbb{R})$. Then the distribution $$\sum_{\nu\in\mathbb{N}_0}\sum_{\tau\in\mathbb{Z}}\langle f,m_{Q_{\nu\tau}}\rangle m_{Q_{\nu\tau}}(x)$$ belongs to ${B}_{pq}^{s}(\mathbb{R},w)$ if $\|{\lambda}\|_{{b}_{pq}^{s}(w)}<\infty$ with
$\{\lambda_{\nu\tau}\}=\{2^{\nu/2}\langle f,m_{Q_{\nu\tau}}\rangle\}$. Moreover, 
$$\|f\|_{B_{pq}^{s}(\mathbb{R},w)}\lesssim \|{\lambda}\|_{{b}_{pq}^{s}(w)}.$$
\end{theorem}

\section {Spline wavelet bases of natural and fractional orders} \label{SP}
We start from the notion of B--splines of natural \cite{Chui} and fractional orders \cite{UB}.

Put $B_0=\chi_{[0,1)}$. B--spline of order $n\in\mathbb{N}$ (see Figure 1) is defined by \begin{equation}\label{Bndef}
B_n(x):=(B_{n-1}\ast B_0)(x)=\int_0^1 B_{n-1}(x-t)\,dt=\frac{x}{n}B_{n-1}(x)
+\frac{n+1-x}{n}B_{n-1}(x-1).
\end{equation} 
\begin{figure*}[h]
\includegraphics[width=.3\textwidth]{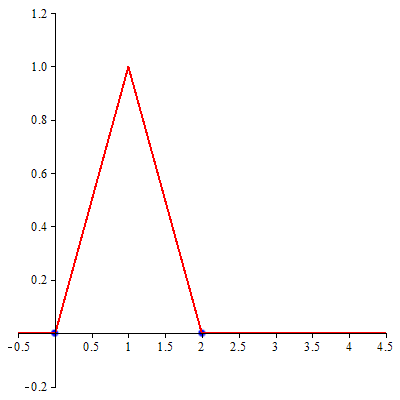}\hfill
\includegraphics[width=.3\textwidth]{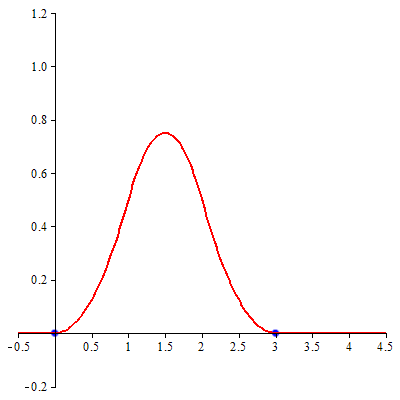}\hfill
\includegraphics[width=.3\textwidth]{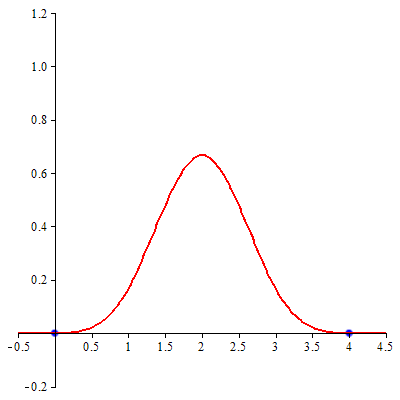}
\caption{Graphs of $B_n$ for $n=1,2,3$.}
\end{figure*}

\noindent{$B_n$ is continuous and $n-$times a.e. differentiable function on $\mathbb{R}$} 
{with ${\rm supp}\,B_n=[0,n+1]$; $B_n(x)>0$ for all $x\in(0,n+1)$ and the restriction of $B_n$ to each $[m,m+1]$,} {$m=0,\ldots,n$, is a polynomial of 
degree $n$.
}

\begin{figure*}[h]
\includegraphics[width=.5\textwidth]{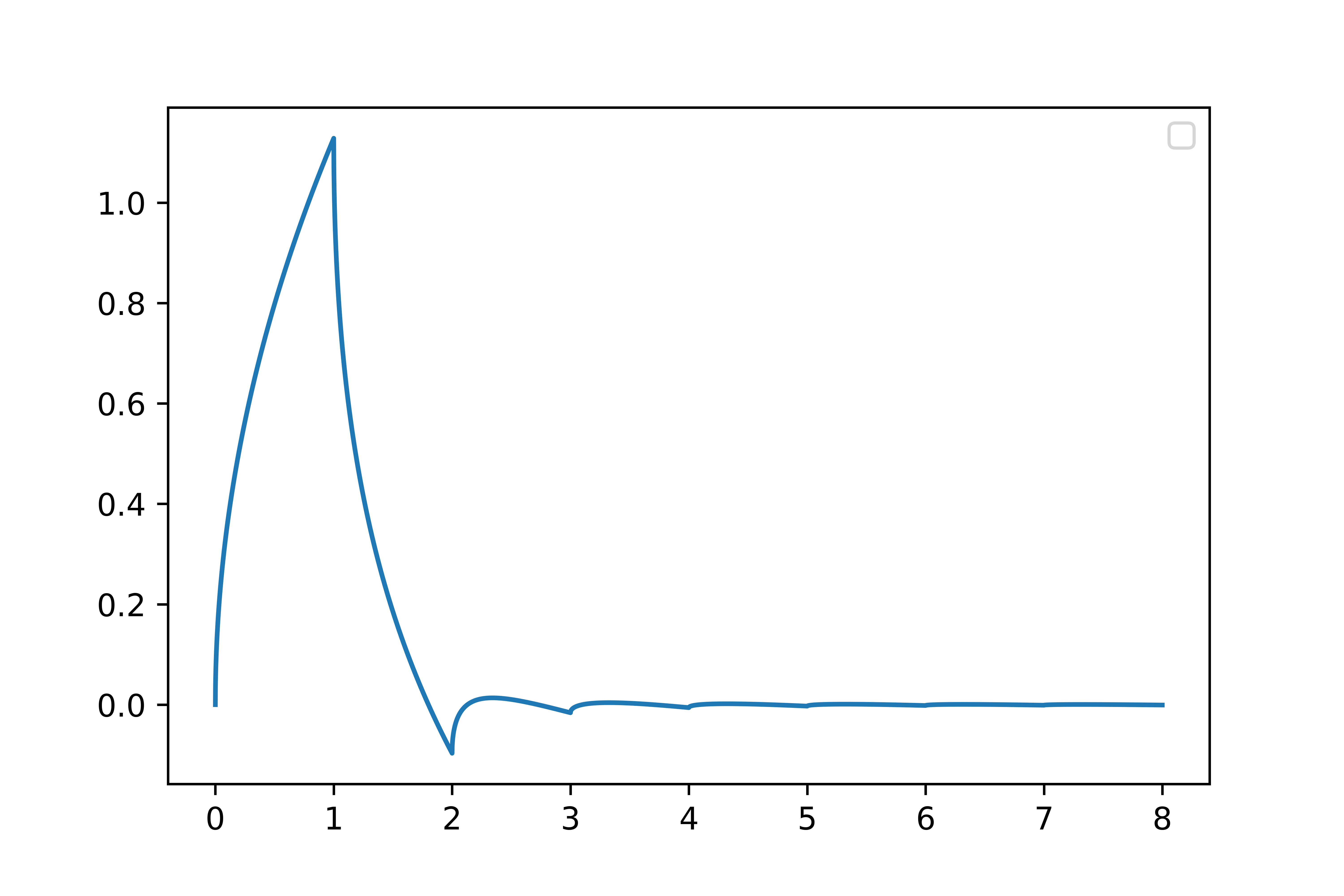}\hfill
\includegraphics[width=.5\textwidth]{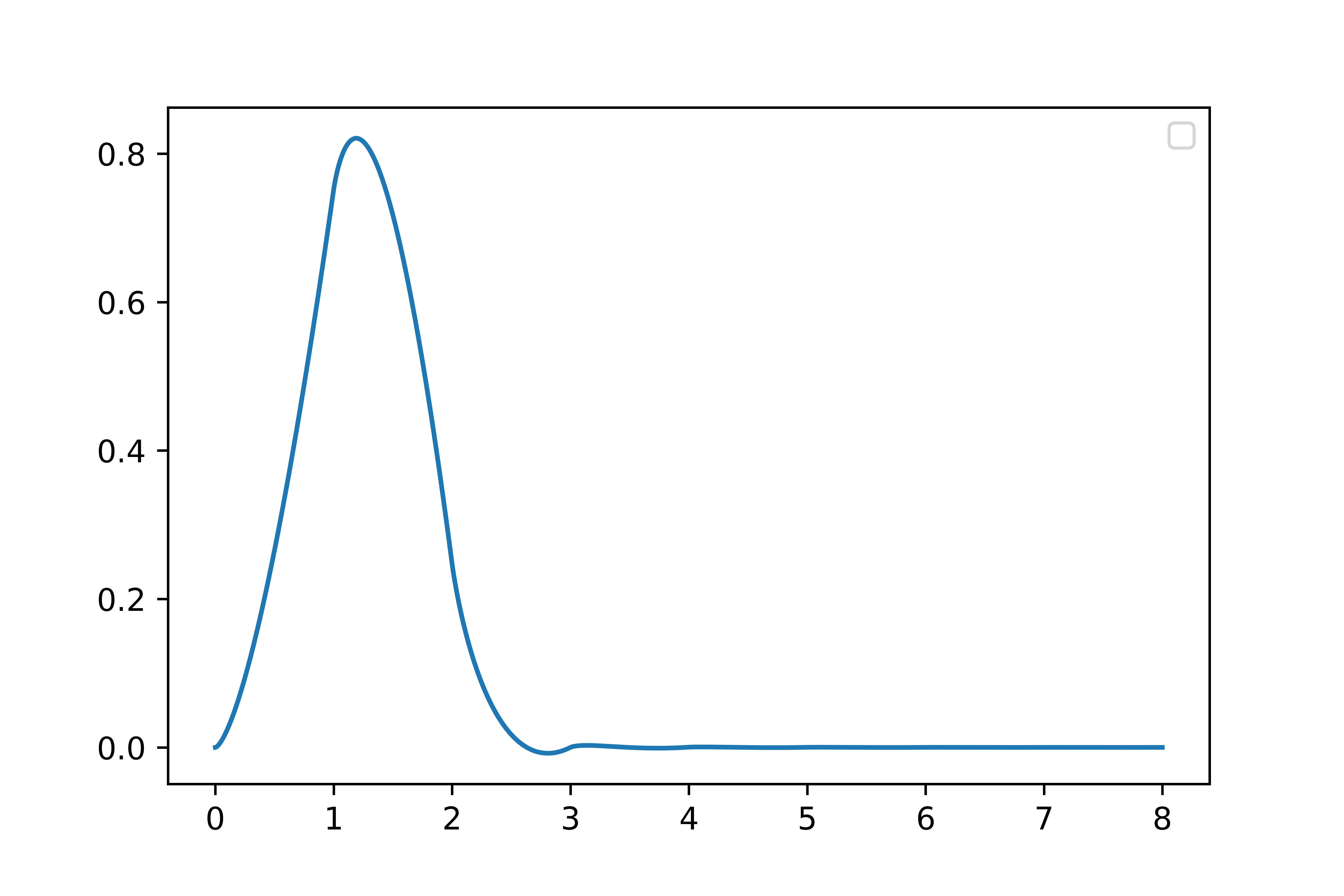}
\caption{Graphs of $\beta_+^{1/2}$ and $\beta_+^{3/2}$.}
\end{figure*}
B--splines were generalised to fractional degrees $\alpha>-1$ by M. Unser and T. Blu in \cite{UB,UB1}. Put $$x_+^\alpha:=\begin{cases} x^\alpha, & x\ge 0,\\ 0, & \textrm{otherwise},\end{cases}\qquad\qquad 
x_-^\alpha:=(-x)_+^\alpha, \qquad\qquad |x|_*^\alpha:=\begin{cases} \frac{|x|^\alpha}{-2\sin(\pi\alpha/2)}, & \alpha\quad\textrm{not even},\\
\frac{x^{\alpha}\log x}{(-1)^{\alpha/2+1}\pi}, & \alpha\quad\textrm{even},\end{cases}$$
{and remind} that $\binom{u}{v}=\frac{\Gamma(u+1)}{\Gamma(v+1)\Gamma(u-v+1)}$. 
Following the terminology adopted in \cite{UB}, fractional causal B--splines of order $\alpha>-1$ (see Figure 2) are defined by
\begin{equation*}
\beta_+^\alpha(x):=\frac{1}{\Gamma(\alpha+1)}\sum_{k\in\mathbb{N}_0}(-1)^k\binom{\alpha+1}{k}(x-k)_+^\alpha,
\end{equation*} 
\begin{figure*}[h]
\includegraphics[width=.5\textwidth]{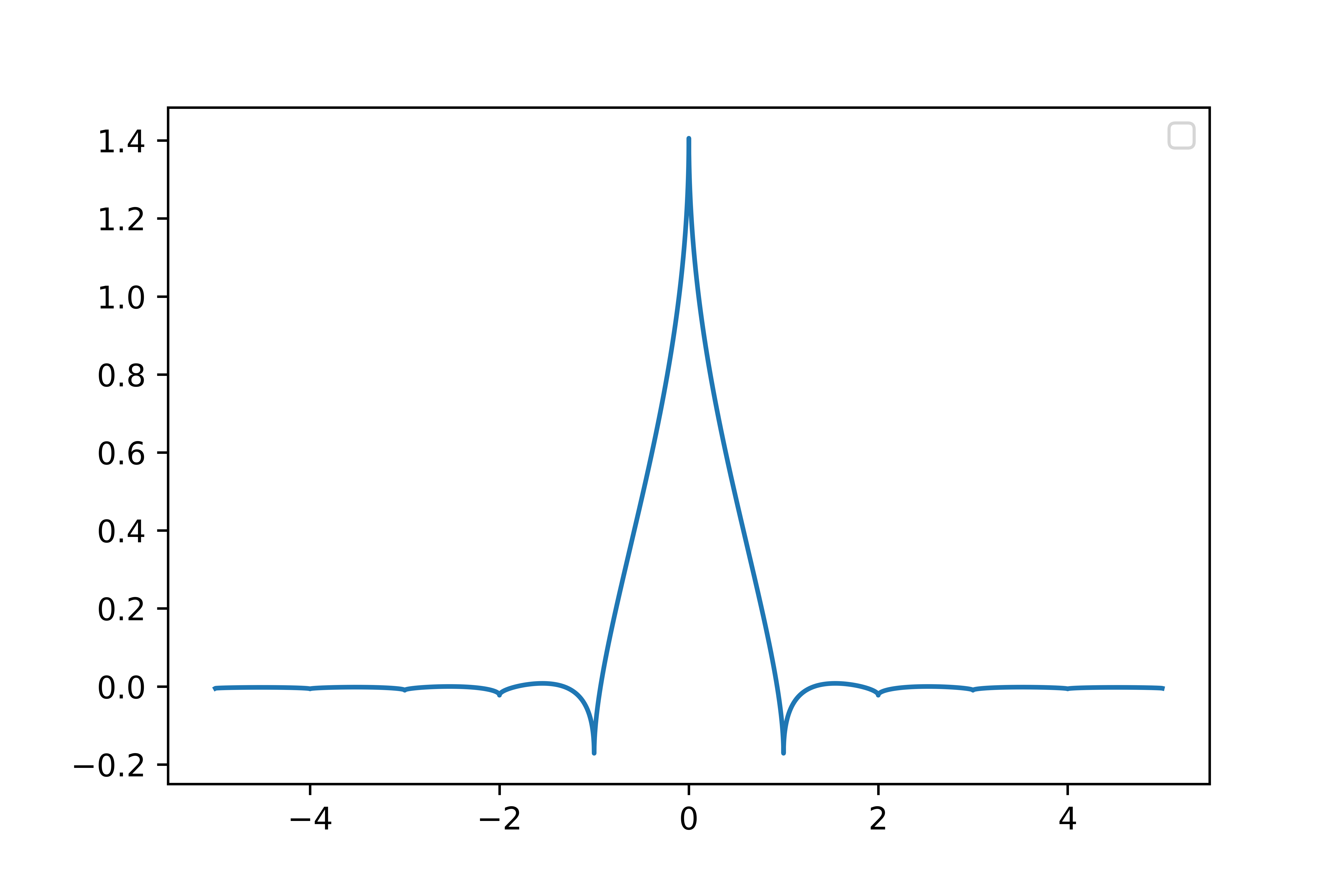}\hfill
\includegraphics[width=.5\textwidth]{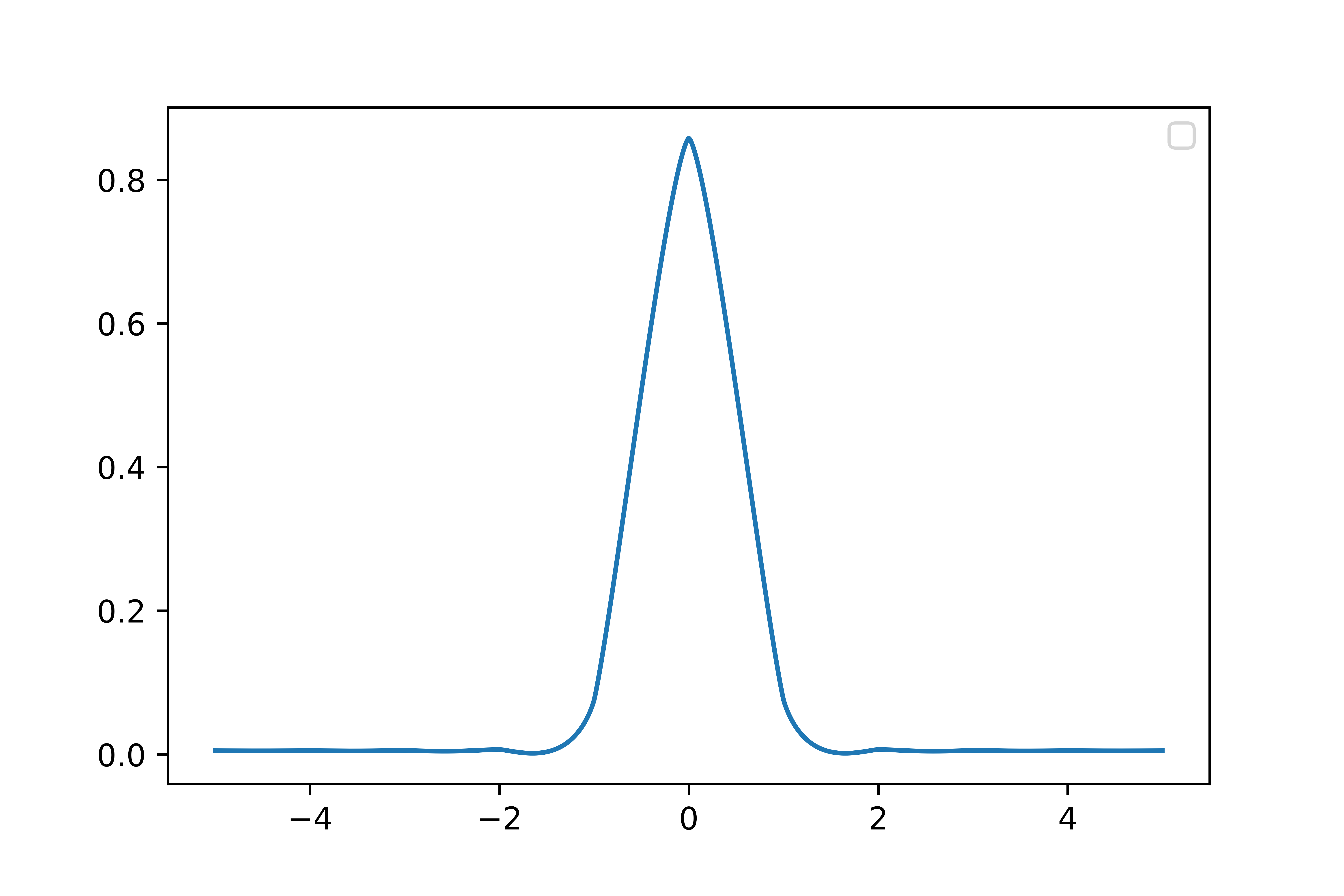}
\caption{Graphs of $\beta_\ast^{1/2}$ and $\beta_\ast^{3/2}$.}
\end{figure*}
the anticausal B--splines of fractional degree $\alpha$ --- 
\begin{equation*}
\beta_-^\alpha(x):=\frac{1}{\Gamma(\alpha+1)}\sum_{k\in\mathbb{N}_0}(-1)^k\binom{\alpha+1}{k}(x+k)_-^\alpha,
\end{equation*} symmetrical B--splines of fractional order $\alpha$ have (see Figure 3) the following form \cite{UB}
\begin{equation*}
\beta_*^\alpha(x):=\beta_+^{(\alpha-1)/2}\ast \beta_-^{(\alpha-1)/2}=
\frac{1}{\Gamma(\alpha+1)}\sum_{k\in\mathbb{Z}}(-1)^{k+1}
\binom{\alpha+1}{k+\alpha/2}|x-k|_*^\alpha.
\end{equation*}

Fractional B--splines are in $L^1(\mathbb{R})$ and in $L^2(\mathbb{R})$ for $\alpha>-1/2$, they decay proportionally to $|x|^{-\alpha-2}$, reproduce the polynomials of degree $[\alpha]+1$. Besides, $\beta^\alpha\in W_2^r$ for all $r<\alpha+1/2$. It holds ${\Gamma(\alpha+1)}\beta_+^\alpha={B_\alpha}$ for $\alpha\in\mathbb{N}$ . If $\alpha\not\in\mathbb{N}$ then B--splines have unbounded supports, there is no symmetry (except $\beta_*^\alpha$) and no positivity in comparison to $\beta_\pm^\alpha$ with $\alpha\in\mathbb{N}$. But, similarly to $B_n$, for $\alpha>-1/2$ the fractional B--spline of degree $\alpha$ generates a Riesz basis in the related subspace $V_0$ (see below) of $L^2(\mathbb{R})$ \cite[Proposition 3.3]{UB}. 

Let ${V}_\nu$, $\nu\in\mathbb{Z}$, denote the $L_2(\mathbb{R})-$closure of the linear span of the system $\bigl\{\beta_{+}^\alpha(2^\nu\cdot-\tau)\colon \tau\in\mathbb{Z}\bigr\}$. The spline spaces ${V}_\nu$, $\nu\in\mathbb{Z}$, constitute multiresolution analysis $\mathrm{MRA}_{\beta_{+}^\alpha}$ of $L_2(\mathbb{R})$ in the sense that \begin{itemize}
\item[{\rm (i)}] $\ldots\subset {V}_{-1}\subset {V}_0\subset {V}_1\subset\ldots$,
\item[{\rm (ii)}] $\mathrm{clos}_{L_2(\mathbb{R})}\Bigl(\bigcup_{\nu\in\mathbb{Z}} {V}_\nu\Bigr)=L_2(\mathbb{R})$,
\item[{\rm (iii)}] $\bigcap_{\nu\in\mathbb{Z}} {V}_\nu=\{0\}$,
\item[{\rm (iv)}] for each $\nu$ the $\bigl\{\beta_+^\alpha(2^\nu\cdot-\tau)\colon\tau\in\mathbb{Z}\bigr\}$ is an unconditional (but not orthonormal) basis of ${V}_\nu$.\end{itemize} 
Further, there are the orthogonal complementary subspaces $\ldots, {W}_{-1},{W}_0,{W}_1,\ldots$ such that 
\begin{itemize}
\item[{\rm (v)}] ${V}_{\nu+1}={V}_\nu\oplus {W}_\nu$ for all $\nu\in\mathbb{Z}$,
where $\oplus$ stands for ${V}_\nu\perp {W}_\nu$ and ${V}_{\nu+1}={V}_\nu+{W}_\nu$. 
\end{itemize} Wavelet subspaces ${W}_\nu$, $\nu\in\mathbb{Z}$, related to the spline $\beta_+^\alpha$, are also generated by some basis functions ({\it wavelets}) in the same manner as the spline spaces ${V}_\nu$, $\nu\in\mathbb{Z}$, are generated by the spline $\beta_+^\alpha$.
Observe that {for any fixed} $\boldsymbol{k}\in\mathbb{Z}$ the system $\bigl\{\beta_{+}^{\alpha,\boldsymbol{k}}(\cdot-\tau):=\beta_{+}^\alpha(\cdot-\boldsymbol{k}-\tau)\colon \tau\in\mathbb{Z}\bigr\}$ generates multiresolution analysis $\mathrm{MRA}_{\beta_{+}^{\alpha,\boldsymbol{k}}}$ of $L_2(\mathbb{R})$, and $\mathrm{MRA}_{\beta_{+}^{\alpha,\boldsymbol{k}}}=\mathrm{MRA}_{\beta_{+}^\alpha}$ for any $\boldsymbol{k}\in\mathbb{Z}$. The same is true for $\beta_-^\alpha$ and $\beta_*^\alpha$.

For non--natural $\alpha$, related to the scale functions $\beta_{\pm}^{\alpha}$ wavelet functions $\psi_\pm^\alpha$ were constructed in \cite{UB1}:
 \begin{equation}\label{WFdef+}
\psi_\pm^\alpha(x):=\sum_{k\in\mathbb{Z}}\frac{(-1)^k}{2^\alpha}\sum_{l\in\mathbb{Z}}\binom{\alpha+1}{l}\beta_*^{2\alpha+1}(l+ k -1)\beta_\pm^\alpha(2x-k).
\end{equation} These functions 
have $[\alpha]+1$ vanishing moments, and are the best possibly localised (see Figure 4). \begin{figure*}[h]
\includegraphics[width=.5\textwidth]{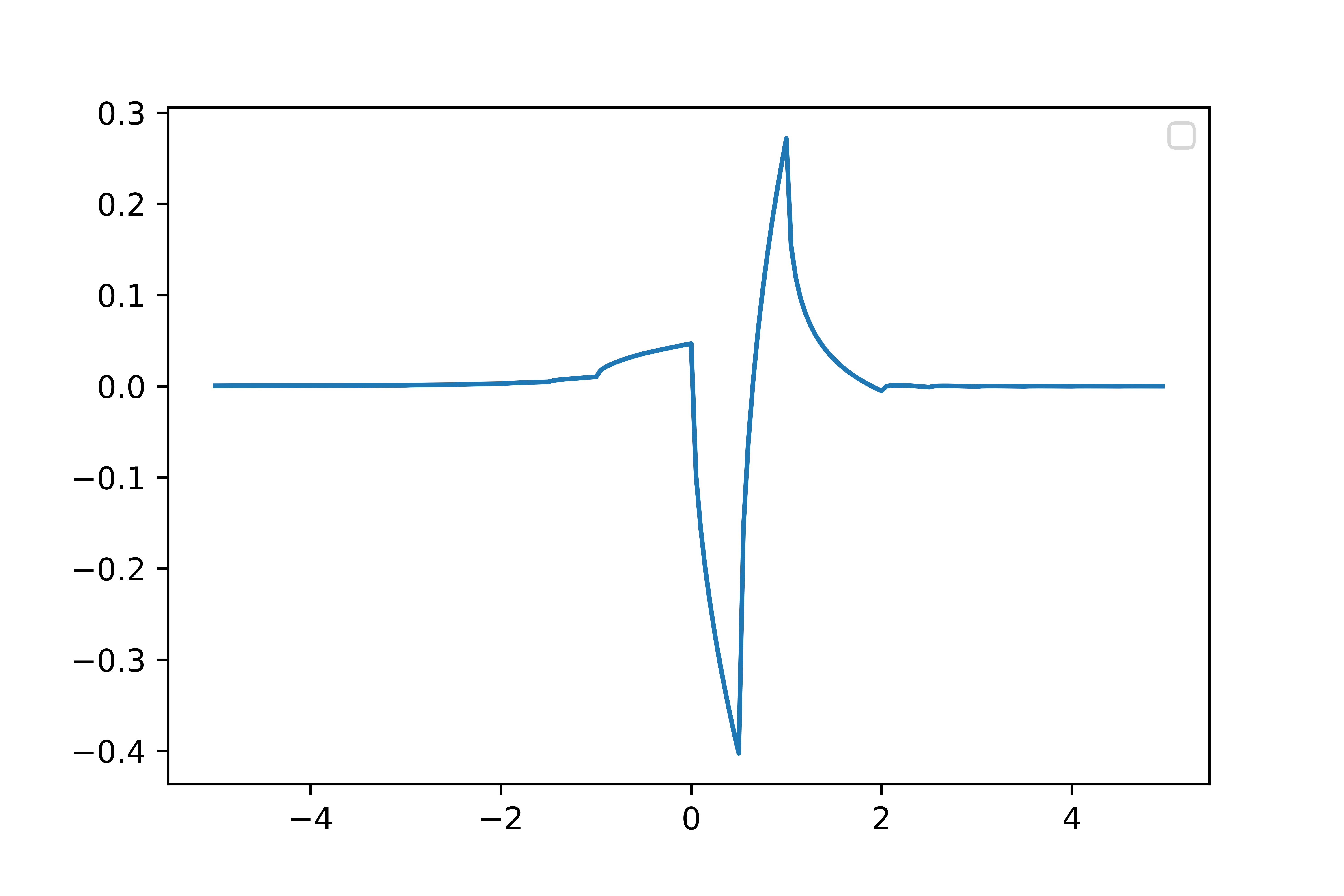}\hfill
\includegraphics[width=.5\textwidth]{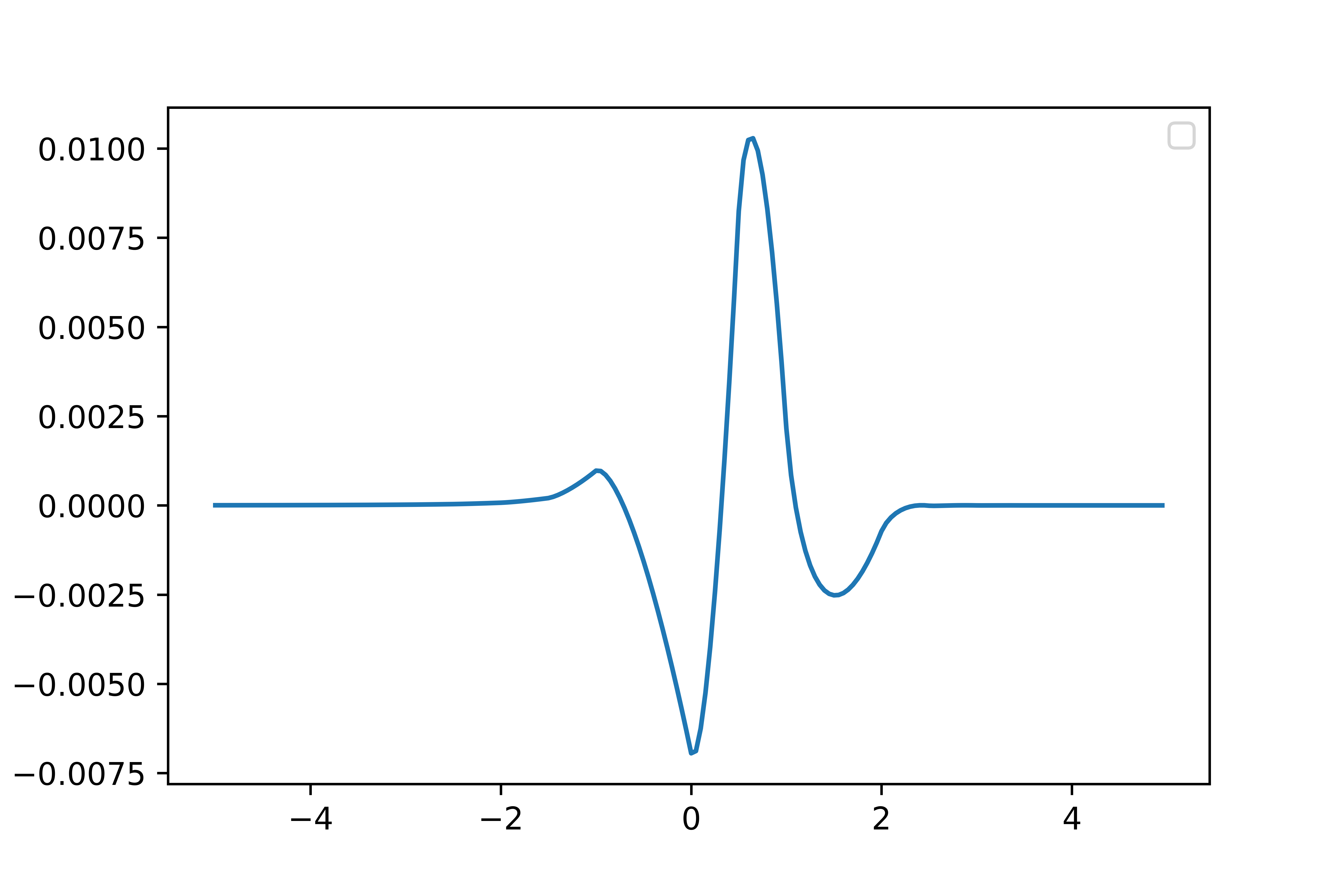}
\caption{Graphs of $\psi_+^{1/2}$ and $\psi_+^{3/2}$.}
\end{figure*}
The limit behaviour of $\psi_\pm^\alpha$ for positive $\alpha\not\in\mathbb{N}$ can be justified basing on \eqref{WFdef+} and \cite[Theorem 3.1]{UB}. 
\begin{theorem}{\cite[Theorem 3.1]{UB}}\label{T1} For all $\alpha>-1$ there exist positive constants $K_\alpha$ and $C_\alpha$ such that for $\beta^\alpha=\beta_+^\alpha$ or $\beta^\alpha=\beta_*^\alpha$ it holds that
\begin{equation}\label{vanish}
|\beta^\alpha(x)|\le\frac{K_\alpha\,\bigl\{\inf_{n\in\mathbb{Z}}|x-n|\bigr\}_*^\alpha+C_\alpha}{1+|x|^{\alpha+2}}.\end{equation}
More precisely, when $\alpha>0$, we have for $x$ tending to $+\infty$
\begin{align*}
\beta_+^\alpha(x)=&\frac{\Gamma(\alpha+2)\sin\pi\alpha}{\pi x^{\alpha+2}}\sum_{n\ge 1}\frac{\mathrm{e}^{2ni\pi x}}{(2ni\pi)^{\alpha+1}}+o\Bigl(\frac{1}{x^{\alpha+2}}\Bigr),
\\
\beta_*^\alpha(x)=&\frac{2\Gamma(\alpha+2)\cos(\pi\alpha/2)}{\pi x^{\alpha+2}}\sum_{n\ge 1}\frac{\cos({2n\pi x})}{(2n\pi)^{\alpha+1}}+o\Bigl(\frac{1}{x^{\alpha+2}}\Bigr).\end{align*}
\end{theorem} 
Analogous behaviour is characteristic of $\beta_-^\alpha$ and $\beta_*^\alpha$, when $x\to-\infty$. 

In view of $\binom{u}{k}=0$ for negative $k\in\mathbb{Z}$,
\begin{align}\label{WFdef'}
\psi_+^\alpha(x)=&\sum_{k\in\mathbb{Z}}\frac{(-1)^k}{2^\alpha}\sum_{l\in\mathbb{N}_0}\binom{\alpha+1}{l}\beta_*^{2\alpha+1}(l+k-1)\beta_+^\alpha(2x- k)\nonumber\\
=&2^{-\alpha}\sum_{m\in\mathbb{Z}}\beta_*^{2\alpha+1}(m-1)\sum_{k\le m}{(-1)^k}\binom{\alpha+1}{m-k}\beta_+^\alpha(2x- k)\nonumber\\=&2^{-\alpha}\sum_{m\in\mathbb{Z}}(-1)^m\beta_*^{2\alpha+1}(m-1)\sum_{n\ge 0}{(-1)^n}\binom{\alpha+1}{n}\beta_+^\alpha(2x- m+n).
\end{align} From here and by Theorem \ref{T1}, by taking into account  
$\binom{u}{r}\approx (-1)^r\frac{\Gamma(u+1)\sin(\pi u)}{\pi r^{u+1}}$ as $r\to+\infty$, we deduce consideration about algebraic rate of decay of $\psi_+^\alpha$ at $\infty$. Wavelets $\psi_-^\alpha$ have similar limiting behaviour. 

Proper translations and dilation of elements of semi--orthogonal spline wavelet system $\{\beta_\pm^\alpha,\psi_\pm^\alpha\}$ constitute a basis in $L^2(\mathbb{R})$ \cite{UB1}.
Spline wavelet systems of natural orders are considered in the next section.

\subsection{Battle--Lemari\'{e} families of natural orders} \label{BL}

Orthogonalisation process of $B$--splines of the form \eqref{Bndef} results in other scaling functions than $B_n$, named after G. Battle \cite{B,B1} and P.G. Lemarie--Rieusset \cite{L}, whose integer translations form orthonormal system within $\mathrm{MRA}_{B_{n,\boldsymbol{k}}}$. 
Constructions of the related orthogonal spline wavelet systems were established in \cite[\S\,2.2]{JMAA}, \cite[\S\,3]{RMC} and \cite[\S\,3]{PSI} (see also \cite{NoSt}).

For each $j=1,\ldots, n$ with $n\in\mathbb{N}$ we define $r_{j}(n):=(2\alpha_j(n)-1)-2\sqrt{\alpha_j(n)(\alpha_j(n)-1)}$ with some particular $\alpha_{j}(n)>1$. Then $r_{j}(n)\in(0,1)$ for all $j=1,\ldots,n$.  Put $\beta_n:=2^n\sqrt{\alpha_1(n)\,r_1(n)\ldots\alpha_n(n)\,r_n(n)}$ and define the $n-$th order Battle--Lemari\'{e} scaling function $\phi_{n,\boldsymbol{k}}$ via its Fourier transform as follows:
\begin{equation}\label{phi_nn}
\hat{\phi}_{n,\boldsymbol{k}}(\omega):={\beta_n\,\hat{B}_{n,\boldsymbol{k}}(\omega)}{\mathbf{A}_n^{-1}(\omega)},\end{equation}
where $\mathbf{A}_n(\omega):=\bigl(1+\mathrm{e}^{i\omega}r_1(n)\bigr)\ldots\bigl
(1+\mathrm{e}^{i\omega}r_n(n)\bigr)$ and parameter $\boldsymbol{k}\in\mathbb{Z}$ is fixed. 
Since 
$$\sum_{m\in\mathbb{Z}}\Bigl|\hat{\phi}_{n,\boldsymbol{k}}(\omega+2\pi m)\Bigr|^2=
\beta_n^2\bigl|\mathbf{A}_n(\omega)\bigr|^{-2}\sum_{m\in\mathbb{Z}}\Bigl|\hat{B}_{n,\boldsymbol{k}}(\omega+2\pi m)\Bigr|^2$$ and (see \cite[\S\,2]{JMAA}) $$\mathbb{P}_{n,\boldsymbol{k}}(\omega):=\sum_{m\in\mathbb{Z}}\Bigl|\hat{B}_{n,\boldsymbol{k}}(\omega+2\pi m)\Bigr|^2=\sum_{m\in\mathbb{Z}}\Bigl|\hat{B}_n(\omega+2\pi m)\Bigr|^2=\frac{1}{\beta_n^2}\bigl|\mathbf{A}_n(\omega)\bigr|^{2},$$ then
$$\sum_{m\in\mathbb{Z}}\Bigl|\hat{\phi}_{n,\boldsymbol{k}}(\omega+2\pi m)\Bigr|^2=
1,$$ that is, {for fixed} $\boldsymbol{k}\in\mathbb{Z}$ the system $\{\phi_{n,\boldsymbol{k}}(\cdot-\tau)\colon \tau\in\mathbb{Z}\}$ forms an orthonormal basis in ${V}_0$ of $\mathrm{MRA}_{{B}_n}$.

It follows from \eqref{phi_nn} and
\begin{equation}\label{ryadok}({\mathrm{e}^{i\omega}r+1})^{-1}=\sum_{l=0}^\infty \left(-r\,\mathrm{e}^{i\omega}\right)^l\qquad (0<r<1)\end{equation} that
\begin{equation}\label{ex2}\hat{\phi}_{n,\boldsymbol{k}}(\omega)=\frac{\beta_n\ \hat{B}_{n,\boldsymbol{k}}(\omega)}{\prod_{j=1}^n \bigl(1+\mathrm{e}^{i\omega}r_j(n)\bigr)}=
\beta_n \prod_{j=1}^n\sum_{l_j=0}^\infty \bigl(-r_j(n)\,\mathrm{e}^{i\omega}\bigr)^{l_j}\,\hat{B}_{n,\boldsymbol{k}}(\omega),\end{equation} that is, 
\begin{equation*}
{\phi}_{n,\boldsymbol{k}}(x)=\beta_n\sum_{l_1\ge 0}\bigl(-r_1(n)\bigr)^{l_1}\ldots\sum_{l_n\ge 0}\bigl(-r_n(n)\bigr)^{l_n}B_{n,\boldsymbol{k}}\bigl(x+l_1+\ldots+ l_n\bigr).\end{equation*}

Denote $\mathscr{A}_n(\omega):=\overline{\mathbf{A}_n(\omega+\pi)}=\bigl(1-\mathrm{e}^{-i\omega}r_1(n)\bigr)\ldots\bigl
(1-\mathrm{e}^{-i\omega}r_n(n)\bigr)$.
For some $\boldsymbol{k},\boldsymbol{s}\in\mathbb{Z}$ the Fourier transform of a wavelet function ${\psi}_{n,\boldsymbol{k},\boldsymbol{s}}$ related to the 
$\phi_{n,\boldsymbol{k}}$ (see Figures 5 and 6) has the form
\begin{equation}\label{Npsi_n}\hat{\psi}_{n,\boldsymbol{k},\boldsymbol{s}}(\omega)=\frac{\beta_n\,\mathrm{e}^{-i\omega\boldsymbol{s}}\,\mathrm{e}^{-i\omega/2}}{2^{n+1}\,\mathrm{e}^{i\pi(n+1+\boldsymbol{k})}}\,
\frac{\mathscr{A}_n(\omega/2)\,\bigl({\mathrm{e}^{i\omega/2}-1}\bigr)^{n+1}}{\mathbf{A}_n(-\omega)\mathbf{A}_n(\omega/2)}\,\hat{B}_{n}(\omega/2).\end{equation}
\begin{figure*}[h]
\includegraphics[width=.5\textwidth]{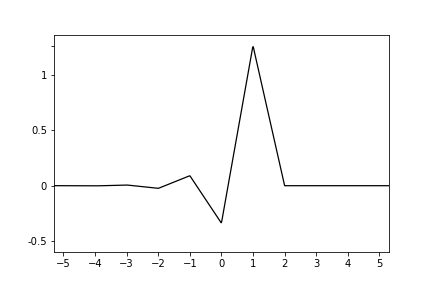}\hfill
\includegraphics[width=.5\textwidth]{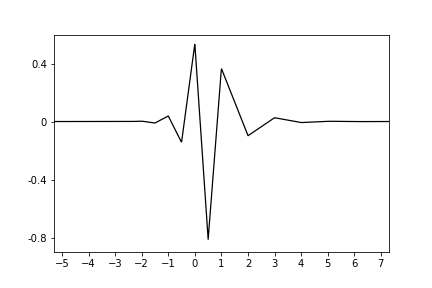}
\caption{Graphs of $\phi_{1,0}$ and $\psi_{1,0,0}$.}
\end{figure*} 

Multiplying the numerator and denominator in \eqref{Npsi_n} by $\mathscr{A}_n(-\omega/2)$ we obtain
\begin{equation}\label{Npsi_nL}\hat{\psi}_{n,\boldsymbol{k},\boldsymbol{s}}(\omega)=\frac{\beta_n\,\mathrm{e}^{-i\omega\boldsymbol{s}}\,\mathrm{e}^{-i\omega/2}}{2^{n+1}\,\mathrm{e}^{i\pi(n+1+\boldsymbol{k})}}\,
\frac{\bigl|\mathscr{A}_n(\pm\omega/2)\bigr|^2\,\bigl({\mathrm{e}^{i\omega/2}-1}\bigr)^{n+1}}{\mathbf{A}_n(-\omega)\mathcal{A}_n(\omega)}\,\hat{B}_{n}(\omega/2)\end{equation}
with $\mathcal{A}_n(\omega):=\mathbf{A}_n(\omega/2)\mathscr{A}_n(-\omega/2)=\bigl(1-\mathrm{e}^{i\omega}r_1^2(n)\bigr)\ldots\bigl
(1-\mathrm{e}^{i\omega}r_n^2(n)\bigr)$. 
\begin{figure*}[h]
\includegraphics[width=.5\textwidth]{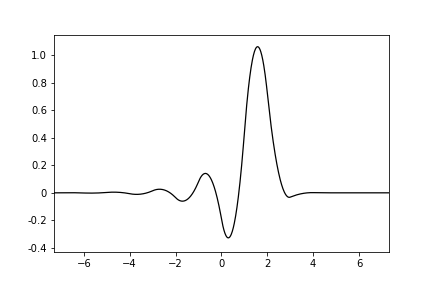}\hfill
\includegraphics[width=.5\textwidth]{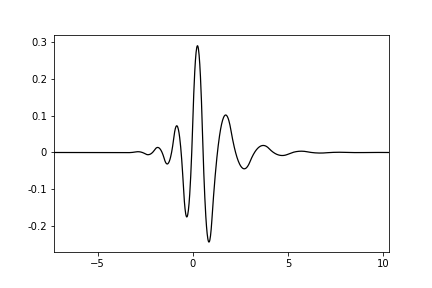}
\caption{Graphs of $\phi_{2,0}$ and $\psi_{2,0,0}$.}
\end{figure*}

Denote $\rho_j(n)=r_j(n)+1/r_j(n)$, $j=1,\ldots,n$. Since
\begin{equation*}\label{modul}
[1-\mathrm{e}^{- i\omega/2}r]\,[1-\mathrm{e}^{i\omega/2}r]=|1-\mathrm{e}^{\pm i\omega/2}r|^2=r\Bigl[\bigl(r+1/r\bigr)-\bigl(\mathrm{e}^{i\omega/2}+\mathrm{e}^{-i\omega/2}\bigr)\Bigr]\qquad(r>0)
\end{equation*} then
\begin{equation*}
\bigl|\mathscr{A}_n(\pm\omega/2)\bigr|^2=\bigl[r_1(n)\ldots r_n(n)\bigr]\sum_{j=0}^n(-1)^j\lambda_j(n)\cos(j\omega/2)\end{equation*} with $\lambda_n(n)=2$, $0<\lambda_j(n)=\lambda_j(\rho_1(n),\ldots,\rho_n(n))$ if $j\not=n$ and even $\lambda_j(n)$ for all $j\not=0$ \cite[\S\,3.1]{PSI}. Letting $\gamma_{n,\boldsymbol{k}}:=\bigl[r_1(n)\ldots r_n(n)\bigr]\beta_n{2^{-n}\cdot(-1)^{n+1+\boldsymbol{k}}}$,
from here and \eqref{Npsi_nL}, we obtain on the strength of \eqref{ryadok}:
\begin{align}\label{Npsihat}\hat{\psi}_{n,\boldsymbol{k},\boldsymbol{s}}(w)=&\frac{\gamma_{n,\boldsymbol{k}}}{2}\biggl\{\sum_{j=0}^n
\frac{\lambda_j(n)}{2(-1)^j}\bigl[\mathrm{e}^{ji\omega/2}+
\mathrm{e}^{-ji\omega/2}\bigr]\biggr\}\sum_{k=0}^{n+1}\frac{(-1)^k(n+1)!}{k!(n+1-k)!}\mathrm{e}^{(n-k)i\omega/2}\nonumber\\
&\times \prod_{j=1}^n\sum_{m_j=0}^\infty \bigl(-r_j(n)\,\mathrm{e}^{- i\omega}\bigr)^{m_j}\sum_{l_j=0}^\infty \bigl(r_j^2(n)\,\mathrm{e}^{ i\omega}\bigr)^{l_j}\ 
\hat{B}_{n}(\omega/2)\ \mathrm{e}^{ -i\omega\boldsymbol{s}}.\end{align}

Orthonormal wavelet systems $\{
{\phi}_{n,\boldsymbol{k}}, {\psi}_{n,\boldsymbol{k},\boldsymbol{s}}\}$ are from $\mathrm{MRA}_{B_n}$  
for any $\boldsymbol{k},\boldsymbol{s}\in\mathbb{Z}$. 
By the substitution $x=\tilde{x}+1/2$ into $B_n$ we arrive to  another type of Battle--Lemari\'{e} wavelet systems of natural orders $\{
\tilde{\phi}_{n,\boldsymbol{k}}, \tilde{\psi}_{n,\boldsymbol{k},\boldsymbol{s}}\}$, which are shifted with respect to $\{
{\phi}_{n,\boldsymbol{k}}, {\psi}_{n,\boldsymbol{k},\boldsymbol{s}}\}$ in $1/2$ to the left. These are from $\mathrm{MRA}_{\tilde{B}_n}$ generated by the shifted $B$--spline $\tilde{B}_n(x):=B_n(x+1/2)$ of natural order $n$. 
Basic properties of Battle--Lemari\'{e} wavelet systems are described in \cite[Proposition 3.1]{RMC}. These systems can be chosen to be $k$--smooth functions if $n\ge k+1$ having exponential decay with decreasing rate as $n$ increases \cite[\S~5.4]{Dau}.

\smallskip
{The Battle--Lemari\'{e} scaling and wavelet functions have unbounded} {supports on $\mathbb{R}$ (see \eqref{ex2} and \eqref{Npsihat}). In what follows we shall 
operate with their localised versions instead (see e.g. \cite[\S\,3.2]{RMC}, \cite[\S\,3.2]{PSI}).}
A localised version of ${\phi}_{n,\boldsymbol{k}}$ can be represented by a function ${\mathbf{\Phi}}_{n,\boldsymbol{k}}$ such that
\begin{equation}\label{phi_fint}
\hat{\mathbf{\Phi}}_{n,\boldsymbol{k}}(\omega)=\hat{\phi}_{n,\boldsymbol{k}}(\omega)\mathbf{A}_n(\omega)=\beta_n\,\hat{B}_{n,\boldsymbol{k}}(\omega).\end{equation} As a localised analogue of ${\psi}_{n,\boldsymbol{k},\boldsymbol{s}}$ we shall use a function $\mathbf{\Psi}_{n,\boldsymbol{k},\boldsymbol{s}}$ satisfying the condition
\begin{equation}\label{Npsihat'}\hat{\mathbf{\Psi}}_{n,\boldsymbol{k},\boldsymbol{s}}(\omega)=
\hat{\psi}_{n,\boldsymbol{k},\boldsymbol{s}}(\omega){\mathbf{A}_n(-\omega)\mathcal{A}_n(\omega)}=
\frac{\gamma_{n,\boldsymbol{k}}}{2\,\mathrm{e}^{i\omega\boldsymbol{s}}}\bigl|\mathscr{A}_n(\pm\omega/2)\bigr|^2\,\bigl|\sum_{k=0}^{n+1}\frac{(-1)^k(n+1)!}{k!(n+1-k)!}\mathrm{e}^{(n-k)i\omega/2}\,\hat{B}_{n}(\omega/2).\end{equation}

On the strength of \eqref{phi_fint} and \eqref{Npsihat'}, the ${\mathbf{\Phi}}_{n,\boldsymbol{k}}$ and $\mathbf{\Psi}_{n,\boldsymbol{k},\boldsymbol{s}}$ are compactly supported.
It also holds: \begin{equation}\label{Fu0}\bigl|\hat{\mathbf{\Phi}}_{n,\boldsymbol{k}}(\omega)\bigr|
=\bigl|\hat{\phi}_{n,\boldsymbol{k}}(\omega){\mathbf{A}_n(\omega)}\bigr|
=\beta_n\bigl|\hat{B}_{n,\boldsymbol{k}}(\omega)\bigr|,\end{equation}\begin{equation}\label{Fu}
\bigl|\hat{\mathbf{\Psi}}_{n,\boldsymbol{k},\boldsymbol{s}}(\omega)\bigr|
=\bigl|\hat{\psi}_{n,\boldsymbol{k},\boldsymbol{s}}(\omega){\mathbf{A}_n(\omega)}\mathcal{A}_n(\omega)\bigr|=\frac{|{\gamma}_{n,\boldsymbol{k}}|}{2}\bigl|\mathscr{A}_{n}(\pm\omega/2)\bigr|^2
\bigl|\mathrm{e}^{i\omega/2}-1\bigr|^{n+1}\bigl|\hat{B}_{n,\boldsymbol{s}}(\omega/2)\bigr|.\end{equation}  
It follows from \eqref{phi_fint} and \eqref{Npsihat'}{, respectively,} that
\begin{equation*}\label{dr}\mathbf{\Phi}_{n,\boldsymbol{k}}(x)={\beta}_{n}\,B_{n,\boldsymbol{k}}(x),\end{equation*}
\begin{equation}\label{psya}
\mathbf{\Psi}_{n,\boldsymbol{k},\boldsymbol{s}}(x)=\frac{{\gamma}_{n,\boldsymbol{k}}}{2^{n}}
\Bigl[
\sum_{j=0}^{n}\frac{\lambda_{j}(n)}{2(-1)^{j}}
\bigl[B_{2n+1}^{(n+1)}\bigl(2(x-\boldsymbol{s})+n+j\bigr)+B_{2n+1}^{(n+1)}\bigl(2(x-\boldsymbol{s})+n-j\bigr)\Bigr]
,\end{equation} where $2^{-n-1}B_{2n+1}^{(n+1)}(2\cdot)$ denotes the $(n+1)-$st order derivative of $B_{2n+1}(2\cdot)$.
Notice that $$\mathbf{\Phi}_{n,\boldsymbol{k}}=\sum_{\kappa=0}^n\boldsymbol{\alpha}'_\kappa\cdot \phi_{n,\boldsymbol{k}-\kappa}\quad\textrm{and}\quad
\mathbf{\Psi}_{n,\boldsymbol{k},\boldsymbol{s}}=\sum_{|\kappa|\le n}\boldsymbol{\alpha}''_\kappa\cdot \psi_{n,\boldsymbol{k},\boldsymbol{s}+\kappa}$$ with some $\boldsymbol{\alpha}'_\kappa$ and $\boldsymbol{\alpha}^{''}_\kappa$ satisfying 
\begin{equation}\label{positive0}\sum_{\kappa=0}^n\boldsymbol{\alpha}'_\kappa=\bigl(1+r_1(n)\bigr)\ldots\bigl(1+r_n(n)\bigr):=\mathbf{\Lambda}'_n>0,\end{equation}
\begin{equation}\label{positive}\sum_{|\kappa|\le n}\boldsymbol{\alpha}''_\kappa=\bigl(1+r_1(n)\bigr)\bigl(1-r_1^2(n)\bigr)\ldots\bigl(1+r_n(n)\bigr)\bigl(1-r_n^2(n)\bigr):=\mathbf{\Lambda}''_n>0\end{equation} (see \eqref{phi_fint}, \eqref{Npsihat'} and definitions of $\mathbf{A}_n(\omega)$ and $\mathscr{A}_n(\omega)$).
The $\mathbf{\Phi}_{n,\boldsymbol{k}}$ and $\mathbf{\Psi}_{n,\boldsymbol{k},\boldsymbol{s}}$ are compactly supported with \begin{equation*}
\textrm{supp}\,\mathbf{\Phi}_{n,\boldsymbol{k}}=[\boldsymbol{k},\boldsymbol{k}+n+1]\quad\textrm{and}\quad\textrm{supp}\,\mathbf{\Psi}_{n,\boldsymbol{k},\boldsymbol{s}}=[\boldsymbol{s}-n/2,\boldsymbol{s}+3n/2+1].\end{equation*}

The functions $\mathbf{\Phi}_{n,\boldsymbol{k}}$ and $\mathbf{\Psi}_{n,\boldsymbol{k},\boldsymbol{s}}$ are finite linear combinations of integer translations of $\phi_{n,\boldsymbol{k}}$ and $\psi_{n,\boldsymbol{k},\boldsymbol{s}}$, respectively, which are elements of the same orthonormal basis in $\textrm{MRA}_{B_n}$ of $L_2(\mathbb{R})$. On the strength of \eqref{Fu0} the system $\{\mathbf{\Phi}_{n,\boldsymbol{k}}(\cdot-\tau)\colon \tau\in\mathbb{Z}\}$ forms a Riesz basis in the subspace ${V}_0\subset L_2(\mathbb{R})$ related to $\textrm{MRA}_{B_n}$. 
At the same time, integer translates of $\mathbf{\Psi}_{n,\boldsymbol{k},\boldsymbol{s}}$ form a Riesz basis in ${W}_0\subset L_2(\mathbb{R})$ related to $\textrm{MRA}_{B_n}$. The both facts are confirmed by the forms of $|\hat{\mathbf{\Phi}}_{n,\boldsymbol{k}}|$ and $|\hat{\mathbf{\Psi}}_{n,\boldsymbol{k},\boldsymbol{s}}|$ (see \eqref{Fu0}, \eqref{Fu}). Similarly to the situation with ${\phi}_{n,\boldsymbol{k}}$ and ${\psi}_{n,\boldsymbol{k},\boldsymbol{s}}$, instead of $\{\mathbf{\Phi}_{n,\tilde{\boldsymbol{k}}},\mathbf{\Psi}_{n,\boldsymbol{k},\tilde{\boldsymbol{s}}}\}$ one can operate with 
the localised systems $\{\tilde{\mathbf{\Phi}}_{n,\boldsymbol{k}},\tilde{\mathbf{\Psi}}_{n,\boldsymbol{k},\boldsymbol{s}}\}$ related to $\{\tilde{\phi}_{n,{\boldsymbol{k}}},\tilde{\psi}_{n,\boldsymbol{k},{\boldsymbol{s}}}\}$ and shifted in $1/2$ to the right with respect to $\{\mathbf{\Phi}_{n,\boldsymbol{k}},\mathbf{\Psi}_{n,\boldsymbol{k},\boldsymbol{s}}\}$. For more properties of the localised analogues of ${\phi}_{n,\boldsymbol{k}}$ and ${\psi}_{n,\boldsymbol{k},\boldsymbol{s}}$ one can consult \cite[\S\,3.2]{PSI}.

\subsection{Spline wavelet bases of fractional orders} \label{FR}
As it was already mentioned before, spline wavelet systems $\{\beta_\pm^\alpha,\psi_\pm^\alpha\}$ of fractional orders $\alpha$ constitute semi--orthogonal bases 
in $L^2(\mathbb{R})$ \cite{UB1}. For $\alpha>-1/2$ integer shifts of the scaling function $\beta_+^\alpha$ (or $\beta_-^\alpha$) form a basis in 
$V_0\in L^2(\mathbb{R})$ related to $\mathrm{MRA}_{\beta_{+}^\alpha}$ (or $\mathrm{MRA}_{\beta_{-}^\alpha}$). Simultaneously, proper translations of dyadic 
dilations of the wavelet function $\psi_+^\alpha$ (or $\psi_-^\alpha$) constitute a basis in $W_d\in L^2(\mathbb{R})$, $d\in\mathbb{N}_0$, within the same multi--resolution analysis. It was mention in \cite[Conclusion]{UB1} that fractional wavelet filters decay reasonably fast for $\alpha>0$. This fact restricts our consideration of $\{\beta_\pm^\alpha,\psi_\pm^\alpha\}$ to positive $\alpha$ only. Besides, for our purposes, instead of $\psi_\pm^\alpha$, we will operate with non--degenerate finite linear combinations of $\psi_\pm^\alpha$ of the following forms
\begin{equation}\label{BigPsi}{\mathit{\Psi}}_{\pm}^\alpha(x):=\sum_{j=0}^{n}\frac{\lambda_{j}(n)}{2(-1)^{j}}\bigl[\psi_\pm^\alpha(x+n+j)
+\psi_\pm^\alpha(x+n-j)\bigr]\qquad (n\in\mathbb{N}),\end{equation}
that is, such that (see \eqref{psya}) \begin{equation*}
\hat{\mathit{\Psi}}_{\pm}^\alpha(\omega)=\psi_\pm^\alpha(\omega)\bigl|\mathscr{A}_n(\pm\omega/2)\bigr|^2.\end{equation*} Since $0<r_j(n)<1$ for all $j=1,\ldots,n$, it holds $0<\bigl|\mathscr{A}_n(\pm\omega/2)\bigr|\le 1$. Therefore, proper shifts of ${\mathit{\Psi}}_{\pm}^\alpha$ are bases in related $W_0$, and the systems $\{\beta_\pm^\alpha,\mathit{\Psi}_\pm^\alpha\}$ constitute semi--orthogonal bases of $L^2(\mathbb{R})$. 

We need to confirm that the systems $\{\beta_\pm^\alpha,\mathit{\Psi}_\pm^\alpha\}$ generate families of smooth molecules for $B_{pq}^s(\mathbb{R},w)$ with $1\le p<\infty$, $0<q\le\infty$, $s\in\mathbb{R}$ and $w\in\mathscr{A}_\infty$ having $r_w$ as in \eqref{param}. To do this we define $J$ as in (M.ii), put $N$ as in (M.iii) and fix some $\delta$ as in (M.i). The choice of $M>J$ from (M.ii) will depend on $\alpha$.

Let us start from $\nu=0$. There exists $c_0>0$ such that $\{c_0\beta_+^\alpha(\cdot-\tau)\}_{\tau\in\mathbb{Z}}$ are ($\delta,M,N$)-- smooth molecules $m_{Q_{0\tau}}$. Indeed, on the strength of Theorem \ref{T1}, for each $\tau$ the condition (M3*) is satisfied for $\gamma=0$ and $M\le\alpha+2$. This means that (M3*) is verified in the case, when the biggest $\gamma$, which can be chosen for given $s$, is zero. To check (M3*) for $0<\gamma\le [s]$ in the case $1\le[s]<\alpha$, we notice that
\begin{equation}\label{26_01}\binom{\alpha+1}{k}=\binom{\alpha}{k}+\binom{\alpha}{k-1}=\ldots=\sum_{j=0}^\gamma\binom{\gamma}{j}\binom{\alpha+1-\gamma}{k-j}.\end{equation} Therefore, \begin{equation}\label{26_02} D^\gamma\beta_+^\alpha(x-\tau)=\sum_{j=0}^\gamma(-1)^j\binom{\gamma}{j}\beta_+^{\alpha-\gamma}(x-\tau-j).\end{equation} For simplicity, assume $\tau=0$. Then
\begin{equation}\label{26_03}D^\gamma\beta_+^\alpha(x)=\begin{cases}\beta_+^{\alpha-\gamma}(x), & 0\le x\le 1,\\
\beta_+^{\alpha-\gamma}(x)-\binom{\gamma}{1}\beta_+^{\alpha-\gamma}(x-1), & 1\le x\le 2,\\\ldots\\
\sum_{j=0}^\gamma(-1)^j\binom{\gamma}{j}\beta_+^{\alpha-\gamma}(x-\tau-j), & x\ge \gamma.\end{cases}\end{equation} By this and from Theorem \ref{T1} we obtain (M3*) for $c_0\beta_+^\alpha$ with  $0<c_0=c_0(\gamma)<1$ and $0<M\le \alpha-\gamma+2$.

To verify (M4*) for $c_0\beta_+^\alpha(\cdot-\tau)$ with $\tau=0$ in the case $\gamma=0$ we begin from $x$ and $y$ satisfying $|x-y|<1$ and assume, for simplicity, that $x>y$.
If $0<x\le 1$ then (M4*) follows from the estimate $$\beta_+^\alpha(x)-\beta_+^\alpha(y)=x^\alpha-y^\alpha\le \max\{\alpha,1\}x^{\alpha-1}(x-y)\lesssim (x-y)^\delta$$ based on \cite[p. 139]{Prokh}. If $0<y\le 1<x<2$ then $$\bigl|\beta_+^\alpha(x)-\beta_+^\alpha(y)\bigr|=\bigl|x^\alpha-y^\alpha-(\alpha+1)(x-1)^\alpha\bigr|
\le \max\{\alpha,1\}x^{\alpha-1}(x-y)+(\alpha+1)(x-y)^\alpha\lesssim \frac{(x-y)^\delta}{(1+y)^M}.$$ This entails the required property, by letting $z=x-y$ and taking the supremum over all $|z|\le|x-y|$. Analogously to the case $0<x\le 1$ we confirm (M4*) for $1<y<x\le 2$, by applying the estimate:
\begin{multline*}\bigl|\beta_+^\alpha(x)-\beta_+^\alpha(y)\bigr|=\bigl|x^\alpha-y^\alpha-(\alpha+1)\bigl[(x-1)^\alpha-(y-1)^\alpha\bigr]\bigr|\\\le \max\{\alpha,1\}\bigl[x^{\alpha-1}+(\alpha+1)(x-1)^{\alpha-1}\bigr](x-y)
\le (\alpha+1)
\bigl[x^{\alpha-1}+(x-1)^{\alpha-1}\bigr](x-y)\lesssim (x-y)^\delta.\end{multline*} The case $1<y\le 2<x<3$ can be verified basing on the inequality
\begin{multline*}\bigl|\beta_+^\alpha(x)-\beta_+^\alpha(y)\bigr|=\bigl|x^\alpha-y^\alpha-(\alpha+1)\bigl[(x-1)^\alpha-(y-1)^\alpha\bigr]+\alpha(\alpha+1)(x-2)^\alpha\bigr|\\
\lesssim \bigl[x^{\alpha-1}+(x-1)^{\alpha-1}+(x-y)^{\alpha-1}\bigr](x-y).\end{multline*} For $2<y<x$ we can write by \eqref{26_01} and in view of $\beta_+^\alpha$ is $\alpha-$H\"{o}lder continuous for each $\alpha>-1$: $$\beta_+^\alpha(x)-\beta_+^\alpha(y)=\alpha \int_y^x\bigl[\beta_+^{\alpha-1}(t)-\beta_+^{\alpha-1}(t-1)\bigr]\,dt.$$
Then, on the strength of Theorem \ref{T1},
$$\bigl|\beta_+^\alpha(x)-\beta_+^\alpha(y)\bigr|\lesssim \int_y^x\biggl[\frac{1}{1+t^{\alpha+1}}+\frac{1}{1+(t-1)^{\alpha+1}}\biggr]\,dt\le (x-y)
\biggl[\frac{1}{1+y^{\alpha+1}}+\frac{1}{1+(y-1)^{\alpha+1}}\biggr].$$ Since $(y-1)/y\ge 1/2$ for $y>2$ then $\frac{1}{1+(y-1)^{\alpha+1}}\lesssim \frac{1}{1+y^{\alpha+1}}$, and now the property (M4*) for the case $|x-y|<1$ follows by letting $z=x-y$ and taking the supremum over all $|z|\le|x-y|$. 

If $x-y\ge 1$ then, from  Theorem \ref{T1}, we obtain with some proper $c_0$
$$c_0\bigl|\beta_+^\alpha(x)-\beta_+^\alpha(y)\bigr|\le 2^{-1} \Bigl[(1+x)^{-\alpha-2}+(1+y)^{-\alpha-2}\Bigr]\le (1+y)^{-\alpha-2}.$$ Letting $z=x-y$ and taking the supremum over $|z|\le|x-y|$ we arrive to the estimate
\begin{equation*}
c_0\bigl|\beta_+^\alpha(x)-\beta_+^\alpha(y)\bigr|\le (1+y)^{-\alpha-2}\le (x-y)^\delta\sup_{|z|\le|x-y|}\bigl(1+|x-z|\bigr)^{-\alpha-2}.\end{equation*} 

For $\gamma\ge 1$ the property (M4*) can be verified by using \eqref{26_01}--\eqref{26_03} supplied $0<M\le\alpha-\gamma+1$.

For $\nu\in\mathbb{N}$ one can check suitability of $\{c\,2^{\nu/2}\mathit{\Psi}_+^\alpha(2^\nu\cdot-\tau)\}_{\tau\in\mathbb{Z}}$ with $0<c=c([s])<1$ to 
(M2)--(M4) similarly to how it was above done for $\{c_0\beta_+^\alpha(\cdot-\tau)\}_{\tau\in\mathbb{Z}}$, taking into account the algebraic rate of decay of the coefficients in \eqref{WFdef'} respectively to $\beta_+^\alpha$. The (M1) is satisfied for $N\le [\alpha]+1$. Observe that (M2) requires $0<M\le \alpha+2+\min\{0,s\}$. 
Summing up, one should choose $0<M\le \begin{cases}\alpha+1-[s], &s\ge  -1,\\ \alpha+2+s, & s<-1.\end{cases}$

\section{Spline wavelet decomposition in $B_{pq}^s(\mathbb{R},w)$ with $w\in\mathscr{A}_\infty$}\label{DEC}

Fix $\alpha>0$, $\boldsymbol{k}$ and $\boldsymbol{s}$.  
As in \S\,3, let ${V}_{\nu}$, $\nu\in\mathbb{N}_0$, denote $\mathrm{MRA}_{{\beta}_\pm^\alpha}$ of the space $L_2(\mathbb{R})$. Put
\begin{equation}\label{vazhnoo}\widetilde{\mathbf{\Phi}}(x):=\begin{cases}\mathbf{\Phi}_{\alpha,\boldsymbol{k}}(x)/\mathbf{\Lambda}'_{\alpha}, &\alpha\in\mathbb{N},\\
c_0\,{\beta}_\pm^{\alpha,\boldsymbol{k}}(x), &\alpha\not\in\mathbb{N}\end{cases}\quad\textrm{and}\quad \widetilde{\mathbf{\Psi}}(x):=\begin{cases}\mathbf{\Psi}_{\alpha,\boldsymbol{k},\boldsymbol{s}}(x)/\mathbf{\Lambda}''_{\alpha}, &\alpha\in\mathbb{N},\\
c\,\mathit{\Psi}_\pm^{\alpha,\boldsymbol{k},\boldsymbol{s}}(x), &\alpha\not\in\mathbb{N}\end{cases}\end{equation} with $\mathbf{\Lambda}'_{\alpha}$ and $\mathbf{\Lambda}''_{\alpha}$ as in \eqref{positive0} and \eqref{positive}. Spline wavelet system $\bigl\{\widetilde{\mathbf{\Phi}},\widetilde{\mathbf{\Psi}}\bigr\}$ of order $\alpha$ constitutes a (semi--orthogonal) Riesz basis (in ${V}_0$ and ${W}_0$, respectively) of order $\alpha$. 
For $x\in\mathbb{R}$ we denote \begin{equation}\label{ForRepr'}\widetilde{\mathbf{\Phi}}_{\tau}(x):=\widetilde{\mathbf{\Phi}}(x-\tau)\quad\textrm{and}\quad \widetilde{\mathbf{\Psi}}_{\nu\tau}(x):=2^{\nu/2}\widetilde{\mathbf{\Psi}}_(2^\nu x-\tau) \qquad  (\tau\in\mathbb{Z},\, \nu\in\mathbb{N}_0).\end{equation}

\smallskip
Characterisation of $B_{pq}^{s}(\mathbb{R},w)$ by spline wavelets of natural orders was performed in \cite[\S~4.3]{PSI} (see also \cite{RMC}). For $w\in\mathscr{A}_\infty$ with ${r}_w$ of the form \eqref{param} we put $\sigma_{p}(w):=\frac{{r}_w}{\min\{p,{r}_w\}}-2+{r}_w$.
\begin{theorem}\label{main'}
Let $0<p<\infty$, $0<q\le\infty$, $s\in\mathbb{R}$ and $w\in\mathscr{A}_\infty$. Let $\widetilde{\mathbf{\Phi}},\,\widetilde {\mathbf{\Psi}}$ be functions satisfying \eqref{vazhnoo} with $\alpha\in\mathbb{N}$. We assume \begin{equation}\label{condB}\alpha\ge \max\Bigl\{0,[s]+1,{\bigr[({r}_w-1)/p-s\bigr]}+1,\bigl[\sigma_{p}(w)-s\bigr]\Bigr\}+1.\end{equation} Then $f\in\mathscr{S}'(\mathbb{R})$ belongs to $B_{pq}^{s}(\mathbb{R},w)$ if and only if it can be represented as
\begin{equation}\label{function}
f=\sum_{\tau\in\mathbb{Z}} \lambda_{0\tau}\widetilde{\mathbf{\Phi}}_\tau+
\sum_{\nu\in\mathbb{N}}\sum_{\tau\in\mathbb{Z}}\lambda_{\nu\tau}2^{-\nu/2}\widetilde{\mathbf{\Psi}}_{(\nu-1)\tau},
\end{equation} where $\lambda\in b_{pq}^{s}(w)$ and the series converges in $\mathscr{S}'(\mathbb{R})$. This representation is unique with
\begin{equation}\label{asty}\lambda_{0\tau}=\langle f,\widetilde{\mathbf{\Phi}}_{\tau}\rangle\quad(\tau\in\mathbb{Z}),\qquad   \lambda_{\nu\tau}=2^{\nu/2}\langle f,\widetilde{\mathbf{\Psi}}_{(\nu-1)\tau}\rangle\quad(\nu\in\mathbb{N},\,\tau\in\mathbb{Z})\end{equation}
and $I\colon f\mapsto\bigl\{\lambda_{d\tau}\bigr\}$ is a linear isomorphism of $B_{pq}^{s}(\mathbb{R},w)$ onto $b_{pq}^{s}(w)$. Besides, \begin{equation*}
\|f\|_{B_{pq}^{s}(\mathbb{R},w)}\approx \|{\lambda}\|_{{b}_{pq}^{s}(w)}.\end{equation*}\end{theorem}

\smallskip For fractional $\alpha>0$ our results in \S\,5 will be based on the following interpretation of Theorem \ref{T0}.
\begin{theorem}\label{FrDec} 
Let $s\in\mathbb{R}$, $1\le p<\infty$, $0<q\le\infty$ and $w\in\mathscr{A}_\infty$. Suppose $\widetilde{\mathbf{\Phi}},\,\widetilde {\mathbf{\Psi}}$ 
are defined by \eqref{vazhnoo} for $\alpha\not\in\mathbb{N}$ and form a family of smooth molecules for ${B}_{pq}^{s}(\mathbb{R},w)$ and 
$f\in\mathscr{S}'(\mathbb{R})$. 
The distribution in the right hand side of \eqref{function} belongs to ${B}_{pq}^{s}(\mathbb{R},w)$
if $\|{\lambda}\|_{{b}_{pq}^{s}(w)}<\infty$ with $\lambda=\{\lambda_{\nu\tau}\}$ of the form \eqref{asty}. Besides,
\begin{equation*}\|f\|_{B_{pq}^{s}(\mathbb{R},w)}\lesssim \|{\lambda}\|_{{b}_{pq}^{s}(w)}.\end{equation*}
\end{theorem}

\section{Norm related inequalities}\label{MR}

Let us begin with examples.
\begin{example}\label{Exam1}
Let $1<p,q<\infty$. Assume $u,w\in\mathscr{A}_\infty$ with $r_u=r_w=1$. Suppose $f\in L_1^\textrm{loc}(\mathbb{R})$ and $f(y)\equiv 0$ if $y\in(-\infty,0)$. Define 
\begin{equation*}
I_{0^+}^{1/3} f(x):=\int_{0}^{x}(x-y)^{-2/3}{f(y)}\,dy\qquad (x>0).
\end{equation*} We show first the validity of the inequality
\begin{equation}\label{ineqEx}\|I_{0^+}^{1/3}f\|_{{B}_{pq}^{0}(\mathbb{R},w)}\lesssim C_{0^+}^{1/3}\|f\|_{{B}_{pq}^{1/3}(\mathbb{R},u)}\quad\textrm{with}\ \ 
C_{0^+}^{1/3}:=\mathscr{M}_{0^+}^{1/3}+\sup_{d\in\mathbb{N}_0}\mathscr{N}_{0^+}^{1/3}({d+1}),\end{equation}
 where \begin{align}\mathscr{M}_{0^+}^{1/3}:=&\sup_{\tau\in \mathbb{Z}^+}\biggl(\sum_{r\ge\tau}(r-\tau+1)^{-2p/3}\int_{Q_{0r}}w\biggr)^{\frac{1}{p}}
\biggl(\sum_{0\le r\le\tau} \Bigl(\int_{Q_{0r}}\bar{u}\Bigr)^{1-p'}\biggr)^{\frac{1}{p'}}<\infty\nonumber\\
&+\sup_{\tau\in \mathbb{Z}^+}\biggl(\sum_{r\ge\tau}\int_{Q_{0r}}w\biggr)^{\frac{1}{p}}\biggl(\sum_{0\le r\le\tau}(\tau-r+1)^{-2p'/3} 
\Bigl(\int_{Q_{0r}}\bar{u}\Bigr)^{1-p'}\biggr)^{\frac{1}{p'}},\label{Md}\\\mathscr{N}_{0^+}^{1/3}(d):=
&\frac{1}{2^{2d/3}}\Biggl[\sup_{\tau\in \mathbb{Z}^+}\biggl(\sum_{r\ge\tau}(r-\tau+1)^{-p/3}\int_{Q_{dr}}w\biggr)^{\frac{1}{p}}\biggl(\sum_{0\le r\le\tau} \Bigl(\int_{Q_{dr}}
\tilde{u}\Bigr)^{1-p'}\biggr)^{\frac{1}{p'}}\nonumber\\&+\sup_{\tau\in \mathbb{Z}^+}\biggl(\sum_{r\ge\tau}\int_{Q_{dr}}w\biggr)^{\frac{1}{p}}\biggl(\sum_{0\le r\le\tau} 
(\tau-r+1)^{-p'/3}\Bigl(\int_{Q_{dr}}\tilde{u}\Bigr)^{1-p'}\biggr)^{\frac{1}{p'}}\Biggr]<\infty\label{Nd}\end{align} with $\tau\in\mathbb{Z}$ and $Q_{dr}=\Bigl[\frac{r-1/2}{2^d},\frac{r+1/2}{2^d}\Bigr]$. Here $\bar{u}\le u$ and $\tilde{u}\le u$ are weight functions.

Observe that $J=1$ for the target space ${B}_{pq}^{0}(\mathbb{R},w)$. Therefore, we can choose $M=2$ and put $N=0$. 
To evaluate from above the left hand side norm in the inequality \eqref{ineqEx} we will use fractional order scaling function $\beta_-^{1+2/3,0}=:\beta_-^{1+2/3}$ 
and related to it wavelet $\mathit{\Psi}_-^{1+2/3}$. On the strength of Theorem \ref{FrDec},
\begin{equation}\label{rysha}\|I_{0^+}^{1/3}f\|_{B^{0}_{pq}(\mathbb{R},w)}\lesssim\|{\lambda}\|_{{b}_{pq}^{0}(w)}= 
\biggl\|\sum_{\nu\in\mathbb{N}_0}\Bigl\|\sum_{\tau\in\mathbb{Z}}
|\lambda_{\nu\tau}|\chi_{Q_{\nu\tau}}\Bigr\|_{L^p(\mathbb{R},w)}^q\biggr\|^{1/q},\end{equation} where (see \eqref{vazhnoo} and \eqref{ForRepr'} with 
$\beta_-^{1+2/3}$ and $\mathit{\Psi}_-^{1+2/3}$)
$$\lambda_{0\tau}=\langle I_{0^+}^{1/3}f,\widetilde{\mathbf{\Phi}}_{\tau}\rangle\quad(\tau\in\mathbb{Z}),\qquad   
\lambda_{\nu\tau}=2^{\nu/2}\langle I_{0^+}^{1/3}f,\widetilde{\mathbf{\Psi}}_{\nu\tau}\rangle\quad(\nu\in\mathbb{N},\,\tau\in\mathbb{Z}).$$ 
Notice that $\langle I_{0^+}^{1/3}f,\widetilde{\mathbf{\Phi}}_{\tau}\rangle=0$ for $\tau\le 0$. For $\tau>0$ we write 
\begin{align*}c_0^{-1}\langle I_{0^+}^{1/3}f,\widetilde{\mathbf{\Phi}}_{\tau}\rangle=
&\int_{\mathbb{R}}\biggl( \int_{0}^{x}\frac{{f(y)}\,dy}{(x-y)^{2/3}}\biggr) 
\beta_-^{1+2/3}(x-\tau)\,dx\\ =&\frac{1}{\Gamma(2+2/3)}\sum_{k\ge 0}(-1)^k\binom{2+2/3}{k}\int_{\mathbb{R}}\biggl( \int_{0}^{x}\frac{{f(y)}\,dy}{(x-y)^{2/3}}\biggr) 
(-x+\tau-k)_+^{1+2/3}\,dx\\=&\frac{1}{\Gamma(2+2/3)}\sum_{k\ge 0}(-1)^k\binom{2+2/3}{k}\int_{0}^{\tau-k}\biggl( \int_{0}^{x}\frac{{f(y)}\,dy}{(x-y)^{2/3}}\biggr) 
(-x+\tau-k)^{1+2/3}\,dx\\=&\frac{1}{\Gamma(2+2/3)}\sum_{k\ge 0}(-1)^k\binom{2+2/3}{k}\int_{0}^{\tau-k}{f(y)}\biggl( \int_{y}^{\tau-k}\frac{(-x+\tau-k)^{1+2/3}\,dx}
{(x-y)^{2/3}}\biggr)dy.\end{align*} 
Since $(-x+\tau-k)^{1+2/3}=(1+2/3)\int_x^{\tau-k}(z-x)^{2/3}\,dz$, then, by the substitution $x=y+s(z-y)$, 
\begin{multline}\label{subst}\frac{3}{5}\int_{y}^{\tau-k}\frac{(-x+\tau-k)^{1+2/3}\,dx}{(x-y)^{2/3}}=\int_{y}^{\tau-k}\biggl(\int_y^{z}\frac{(z-x)^{2/3}\,dx}{(x-y)^{2/3}}\biggr)dz
\\=\int_{y}^{\tau-k}(z-y)\,dz\int_0^1s^{-2/3}(1-s)^{2/3}\,ds=B\Bigl(\frac{1}{3},\frac{5}{3}\Bigr)\int_{y}^{\tau-k}(z-y)\,dz=\frac{1}{2}B\Bigl(\frac{1}{3},\frac{5}{3}\Bigr)
(-y+\tau-k)^2,
\end{multline} where $B\bigl(1/3,5/3\bigr)$ is the Beta--function. We have 
\begin{equation}\label{05_01}\langle I_{0^+}^{1/3}f,\widetilde{\mathbf{\Phi}}_{\tau}\rangle=B\Bigl(\frac{1}{3},
\frac{5}{3}\Bigr)\frac{5 c_0}{6\Gamma(2+2/3)}\sum_{k\ge 0}(-1)^k\binom{2+2/3}{k}\int_{0}^{\tau-k}{f(y)}(-y+\tau-k)^2\,dy.\end{equation} 
Our aim now is to reduce the integral in the right hand side of \eqref{05_01} to $\bigl\langle f(\cdot)B_2(\cdot-\tau+k+3)\bigr\rangle$, where $B_2(y-\tau+k+3)=-\Gamma(3)\beta^{2,\tau-k-3}(y)$ is the B-spline of the second order.  
In order to do this we shall use the Chu--Vandermonde identity (see e.g. \cite[p.15]{SKM}, \cite[pp. 59--60]{Ask} or \cite{Van}):
\begin{equation}\label{ChVan}
\binom{r+s}{k}=\sum_{n=0}^k\binom{r}{n}\binom{s}{k-n},\qquad \textrm{where}\quad r,s\in\mathbb{R}\quad \textrm{and}\quad k\in\mathbb{N}.\end{equation} \label{Page1} On the strength of \eqref{ChVan} with $r=-1/3$ and $s=3$, 
\begin{align*}&\sum_{k\ge 0}(-1)^k\binom{2+2/3}{k}\int_{0}^{\tau-k}{f(y)}(-y+\tau-k)^2\,dy\\
=&\sum_{k\ge 0}(-1)^k\sum_{n=0}^k\binom{-1/3}{n}\binom{3}{k-n}\int_{0}^{\infty}{f(y)}(-y+\tau-k)^2_+\,dy\\=&
\sum_{n\ge 0}\binom{-1/3}{n}\sum_{k=n}^\infty(-1)^k\binom{3}{k-n}\int_{0}^{\infty}{f(y)}(-y+\tau-k)^2_+\,dy
\\=&
\sum_{n\ge 0}(-1)^n\binom{-1/3}{n}\sum_{m\ge 0}(-1)^m\binom{3}{m}\int_{0}^{\infty}{f(y)}(-y+\tau-n-m)^2_+\,dy\\
=&-
\sum_{n=0}^\tau(-1)^n\binom{-1/3}{n}\int_{0}^{\infty}{f(y)}B_2(y-\tau+n+3)^2\,dy,\end{align*}
in view of $\int_0^\infty f(y)B_2(y-\tau+n+3)\,dy=0$ if $n>\tau$ (since $f\equiv 0$ on $\mathbb{R}_-$).
It is known that 
\begin{equation}\label{BinCoef}
(-1)^k\binom{z}{k}=\binom{k-z-1}{k}=\frac{1}{\Gamma(-z)}\frac{1}{(k+1)^{z+1}}\prod_{j=k+1}^\infty\frac{(1+1/j)^{-z-1}}{1-\frac{z+1}{j}}.
\end{equation} Thus, absolute values of $\binom{-1/3}{k}$ can be evaluated from above by ${(k+1)^{-2/3}}$ with some constant. 
And the term for $\nu=0$ in the right hand side of \eqref{rysha} can be estimated as follows, by taking into account \eqref{05_01}:
\begin{align*}
\sum_{\tau\ge 0}\Bigl(\int_{Q_{0\tau}}w\Bigr)^p\bigl|\langle I_{0^+}^{1/3}f,\widetilde{\mathbf{\Phi}}_{\tau}\rangle\bigr|^p
\lesssim& \sum_{\tau\ge 0}\Bigl(\int_{Q_{0\tau}}w\Bigr)^p\biggl[\sum_{k=0}^\tau (k+1)^{-2/3}\Bigl|\int_{0}^{\tau-k}{f(y)}B_2(y-\tau+3+k)\,dy\Bigr|
\biggr]^p\nonumber\\
=&\sum_{\tau\ge 0}\Bigl(\int_{Q_{0\tau}}w\Bigr)^p\biggl[\sum_{n=0}^\tau (\tau-n+1)^{-2/3}\Bigl|\int_{0}^{n}{f(y)}B_2(y-n+3)\,dy\Bigr|\biggr]^p.
\end{align*} Further, on the strength of \cite[Theorem 1.8]{PSU} and in view of \eqref{Md},
\begin{equation}\label{06_07_03}\biggl(\sum_{\tau\ge 0}\Bigl(\int_{Q_{0\tau}}w\Bigr)^p\bigl|\langle I_{0^+}^{1/3}f,\widetilde{\mathbf{\Phi}}_{\tau}\rangle\bigr|^p
\biggr)^{1/p}
\lesssim \mathscr{M}^{1/3}_{0^+}\biggl(\sum_{\tau\ge 0}\Bigl(\int_{Q_{0\tau}}u\Bigr)^p\biggl|\int_{0}^{\tau}{f(y)}B_2(y-\tau+3)\,dy\biggr|^p\biggr)^{1/p}.
\end{equation} We will denote $\int_{0}^{\tau}{f(y)}B_2(y-\tau+3)\,dy=:\langle f,{\mathbf{\Phi}}_{2,\tau-3}\rangle$ in \eqref{06_07_03} for starting to formulate a new basis generated by
the second order scaling ${\mathbf{\Phi}}_{2,-3}/\mathbf{\Lambda}'_2$ and wavelet ${\mathbf{\Psi}}_{2,-3,-6}/\mathbf{\Lambda}^{''}_2$ functions.

To estimate $\lambda_{\nu\tau}$ for $\nu\not=0$ in the right hand side of 
\eqref{rysha} from above by $\langle f,{\mathbf{\Psi}}_{2,-3,-6}/\mathbf{\Lambda}^{''}_2\rangle$ we write, by making use of \eqref{WFdef'},
\begin{align}\label{Proof1}&
2^{5/3}{\Gamma(8/3)}(\langle I_{0^+}^{1/3}f,\psi^{5/3}_-(2^{\nu-1}\cdot-\tau)\rangle\nonumber\\=&{\Gamma(8/3)}
\sum_{m\in\mathbb{Z}}(-1)^m\beta_\ast^{13/3}(m-1)\sum_{n\ge 0}(-1)^n\binom{8/3}{n}\int_{\mathbb{R}}\biggl( \int_{0}^{x}\frac{{f(y)}\,dy}{(x-y)^{2/3}}\biggr) 
\beta_-^{5/3}(2^\nu x-m+n-\tau)\,dx\nonumber\\=&\!\sum_{m\in\mathbb{Z}}(-1)^m\beta_\ast^{13/3}(m-1)\sum_{n\ge 0}(-1)^n\binom{8/3}{n}\nonumber\\
&\times\sum_{l\ge 0}(-1)^l\binom{8/3}{l}\int_{\mathbb{R}}\biggl( \int_{0}^{x}\frac{{f(y)}\,dy}{(x-y)^{2/3}}\biggr) 
(-2^\nu x+m-n-l+\tau)_+^{5/3}\,dx,\end{align}
where, similarly to the case $\nu=0$ (see \eqref{subst}), and under the assumption that $\tau+m>0$,
\begin{align}\label{09_07_01}&\int_{\mathbb{R}}\biggl( \int_{0}^{x}\frac{{f(y)}\,dy}{(x-y)^{2/3}}\biggr) 
(-2^\nu x+m-n-l+\tau)_+^{5/3}\,dx\nonumber\\&=
\int_{0}^{\frac{\tau+m-l-n}{2^\nu}}\biggl( \int_{0}^{x}\frac{{f(y)}\,dy}{(x-y)^{2/3}}\biggr) 
(-2^\nu x+m-n-l+\tau)^{5/3}\,dx\nonumber\\&=
{2^{-\nu}}\int_{0}^{{\tau+m-l-n}}\biggl( \int_{0}^{2^{-\nu}z}\frac{{f(y)}\,dy}{(2^{-\nu}z-y)^{2/3}}\biggr) 
(-z+m-n-l+\tau)^{5/3}\,dz\nonumber\\&={2^{-\frac{4\nu}{3}}}\int_{0}^{{\tau+m-l-n}}\biggl( \int_{0}^{z}\frac{{f(2^{-\nu}t)}\,dt}{(z-t)^{2/3}}\biggr) 
(-z+m-n-l+\tau)^{5/3}\,dz\nonumber\\&={2^{-\frac{4\nu}{3}}}\int_{0}^{{\tau+m-l-n}}{f(2^{-\nu}t)}
\biggl( \int_{t}^{{\tau+m-l-n}}\frac{(-z+m-n-l+\tau)^{5/3}\,dz}
{(z-t)^{2/3}}\biggr)dt\nonumber\\&\simeq
{2^{-\frac{4\nu}{3}}}\int_{0}^{{\tau+m-l-n}}{f(2^{-\nu}t)}\biggl( \int_{t}^{{\tau+m-l-n}}(r-t)\,dr\biggr)dt
\nonumber\\&\simeq{2^{-\frac{4\nu}{3}}}\int_{0}^{{\tau+m-l-n}}{f(2^{-\nu}t)}(-t+\tau+m-l-n)^2\,dt\nonumber\\
&={2^{-\nu/3}}\int_{0}^{\frac{\tau+m-l-n}{2^{\nu}}}{f(y)}(-2^\nu y+\tau+m-l-n)^2\,dy.\end{align} 
We have
\begin{multline*}2^{5/3}{2^{\nu/3}}{\Gamma(2+2/3)}(\langle I_{0^+}^{1/3}f,\psi^{5/3}_-(2^{\nu-1}\cdot-\tau)\rangle\\ \simeq
\sum_{m\in\mathbb{Z}}(-1)^m\beta_\ast^{13/3}(m-1)\sum_{n\ge 0}(-1)^n\binom{8/3}{n}
\sum_{l\ge 0}(-1)^l\binom{8/3}{l}\int_{0}^{\frac{\tau+m-l-n}{2^{\nu}}}{f(y)}(-2^\nu y+\tau+m-l-n)^2\,dy.\end{multline*}
We need to pass now from $\sum_{n\ge 0}(-1)^n\binom{8/3}{n}
\sum_{l\ge 0}(-1)^l\binom{8/3}{l}(-2^\nu y+\tau+m-l-n)^2_+$ to
$$\sum_{n\ge 0}(-1)^n\binom{3}{n}
\sum_{l\ge 0}(-1)^l\binom{3}{l}(-2^\nu y+\tau+m-l-n)^2_+=-4\sum_{n\ge 0}(-1)^n\binom{3}{n}B_2(2^\nu y-\tau-m+n+3).$$ To this end  we write, by using \eqref{ChVan} with $r=s=8/3$,
\begin{align*}&\sum_{n\ge 0}(-1)^n\binom{8/3}{n}
\sum_{l\ge 0}(-1)^l\binom{8/3}{l}(-2^\nu y+\tau+m-l-n)^2_+\\&=
\sum_{n\ge 0}\binom{8/3}{n}
\sum_{k\ge n}(-1)^k\binom{8/3}{k-n}(-2^\nu y+\tau+m-k)^2_+\\&=
\sum_{k\ge 0}(-1)^k(-2^\nu y+\tau+m-k)^2_+
\sum_{k\ge n}\binom{8/3}{n}\binom{8/3}{k-n}=
\sum_{k\ge 0}(-1)^k\binom{16/3}{k}(-2^\nu y+\tau+m-k)^2_+.\end{align*} 
Observe that $\sum_{n\ge 0}(-1)^n\binom{3}{n}B_2(2^\nu y-\tau-m+n+3)$ can be written through iterated differences as
$$\sum_{n\ge 0}(-1)^n\binom{3}{n}B_2(2^\nu y-\tau-m+n+3)=\Delta_{2^{-\nu}}^3B_2(2^\nu y-\tau-m+6)=\Delta_{2^{-\nu}}^6 (-2^\nu y+\tau+m-6)^2_+.$$
Analogously to how that was done for the case $\nu=0$ (see p. \pageref{Page1}), we obtain for $\tau+m>0$
$$\sum_{k\ge 0}(-1)^k\binom{16/3}{k}(-2^\nu y+\tau+m-k)^2_+=4\sum_{k=0}^{\tau+m}(-1)^{k+1}\binom{-2/3}{k}
\Delta_{2^{-\nu}}^3B_2(2^\nu y-\tau-m+k+6)$$ by virtue of \eqref{ChVan} with $s=6$ and $r=-2/3$. 
From here, by using \eqref{BinCoef} with $z=-2/3$ and taking into account \eqref{BigPsi} and \eqref{09_07_01}, we arrive to the estimate
\begin{align}\label{12_07_01}
\sum_{\tau\in\mathbb{Z}}\Bigl(\int_{Q_{\nu\tau}}w\Bigr)^p\bigl|\langle I_{0^+}^{1/3}f,\widetilde{\mathbf{\Psi}}_{(\nu-1)\tau}\rangle\bigr|^p
\lesssim& 2^{-p\nu/3}\sum_{\tau\in\mathbb{Z}}\Bigl(\int_{Q_{\nu\tau}}w\Bigr)^p\biggl[
\sum_{m\in\mathbb{Z}}\bigl|\beta_*^{13/3}(m-1)
\bigr|\sum_{k=0}^{\tau+m} (k+1)^{-1/3}\nonumber\\
&\times\biggl|\int_{0}^{\frac{\tau+m-k}{2^\nu}}{f(y)}\sum_{j=0}^2\frac{\lambda_j(2)}{2(-1)^j}
\Bigl[\Delta_{2^{-\nu}}^3B_2(2^\nu y-\tau-m+k+8+j)\nonumber\\&-\Delta_{2^{-\nu}}^3B_2(2^\nu y-\tau-m+k+8-j)\Bigr]\,dy\biggr|
\biggr]^p,\end{align} where the integral equals to 0 for $m<-\tau$ since $f\equiv 0$ on $\mathbb{R}_-$. Further, in view of \eqref{Muck} and \eqref{vanish},
\begin{equation}\label{12_07_02}\Bigl(\int_{Q_{\nu\tau}}w\Bigr)\bigl|\beta_*^{13/3}(m-1)\bigr|\lesssim \Bigl(\int_{Q_{\nu|m|}}w\Bigr)\bigl(1+|m|\bigr)^{-3-2/3}\lesssim
\Bigl(\int_{Q_{\nu(\tau+m)}}w\Bigr)\bigl(1+|m|\bigr)^{-2-2/3},\end{equation}
where $Q_{\nu|m|}=\Bigl[\min\Bigl\{\frac{\tau}{2^\nu},\frac{\tau+m+1}{2^\nu}\Bigr\},\max\Bigl\{\frac{\tau}{2^\nu},\frac{\tau+m+1}{2^\nu}\Bigr\}\Bigr]$. Denote 
$${\mathbf{\Psi}}_{2,-3,-6}(2^{\nu-1}\cdot-\tau-m+k)=\sum_{j=0}^2\frac{\lambda_j(2)}{2(-1)^j}
\Bigl[\Delta_{2^{-\nu}}^3B_2(2^\nu \cdot-\tau-m+k+8+j)-\Delta_{2^{-\nu}}^3B_2(2^\nu \cdot-\tau-m+k+8-j)\Bigr]$$ by $\bar{\mathbf{\Psi}}_{(\nu-1)(\tau+m-k)}(\cdot)$ 
and continue the estimate \eqref{12_07_01}, by taking into account \eqref{12_07_02}, as follows: 
\begin{align*}
&\sum_{\tau\in\mathbb{Z}}\Bigl(\int_{Q_{\nu\tau}}w\Bigr)^p\bigl|\langle I_{0^+}^{1/3}f,\widetilde{\mathbf{\Psi}}_{(\nu-1)\tau}\rangle\bigr|^p\\&
\lesssim 2^{-p\nu/3}\sum_{\tau\in\mathbb{Z}}\Bigl(\int_{Q_{\nu\tau}}w\Bigr)^p\biggl[
\sum_{m\ge-\tau}\bigl|\beta_*^{13/3}(m-1)
\bigr|\sum_{k=0}^{\tau+m} (k+1)^{-1/3}\biggl|\int_{0}^{\tau+m-k}{f(y)}
\bar{\mathbf{\Psi}}_{(\nu-1)(\tau+m-k)}(y)\,dy\biggr|
\biggr]^p\\&
\lesssim 2^{-p\nu/3}\sum_{\tau\in\mathbb{Z}}\biggl[
\sum_{m\ge-\tau}\bigl(1+|m|\bigr)^{-8/3}\Bigl(\int_{Q_{\nu(\tau+m)}}w\Bigr)
\sum_{k=0}^{\tau+m} (k+1)^{-1/3}\biggl|\int_{0}^{\tau+m-k}{f(y)}
\bar{\mathbf{\Psi}}_{(\nu-1)(\tau+m-k)}(y)\,dy\biggr|
\biggr]^p.\end{align*}
By H\'{o}lder's inequality,
\begin{multline*}
\sum_{m\ge-\tau}\bigl(1+|m|\bigr)^{-8/3}\Bigl(\int_{Q_{\nu(\tau+m)}}w\Bigr)
\sum_{k=0}^{\tau+m} (k+1)^{-1/3}\biggl|\int_{0}^{\tau+m-k}{f(y)}
\bar{\mathbf{\Psi}}_{(\nu-1)(\tau+m-k)}(y)\,dy\biggr|\\\le 
\Bigl(\sum_{m\in\mathbb{Z}}\bigl(1+|m|\bigr)^{-8p'/3}\Bigr)^{1/p'}\biggl(\sum_{m\ge-\tau}\bigl(1+|m|\bigr)^{-8p/3}\Bigl(\int_{Q_{\nu(\tau+m)}}w\Bigr)^p\\\times
\biggr[\sum_{k=0}^{\tau+m} (k+1)^{-1/3}\biggl|\int_{0}^{\tau+m-k}{f(y)}
\bar{\mathbf{\Psi}}_{(\nu-1)(\tau+m-k)}(y)\,dy\biggr|\biggr]^p\biggr)^{1/p}.
\end{multline*}
Thus,
\begin{align*}
&\sum_{\tau\in\mathbb{Z}}\Bigl(\int_{Q_{\nu\tau}}w\Bigr)^p\bigl|\langle I_{0^+}^{1/3}f,\widetilde{\mathbf{\Psi}}_{(\nu-1)\tau}\rangle\bigr|^p\\&
\lesssim 2^{-p\nu/3}\sum_{\tau\in\mathbb{Z}}\sum_{m\ge-\tau}\bigl(1+|m|\bigr)^{-8p/3}
\Bigl(\int_{Q_{\nu(\tau+m)}}w\Bigr)^p
\biggl[\sum_{k=0}^{\tau+m} (k+1)^{-1/3}\biggl|\int_{0}^{\tau+m-k}{f(y)}
\bar{\mathbf{\Psi}}_{(\nu-1)(\tau+m-k)}(y)\,dy\biggr|
\biggr]^p\\&
\lesssim 2^{-p\nu/3}\sum_{m\in\mathbb{Z}}\bigl(1+|m|\bigr)^{-8p/3}\sum_{-m\le \tau}
\Bigl(\int_{Q_{\nu(\tau+m)}}w\Bigr)^p
\biggl[\sum_{k=0}^{\tau+m} (k+1)^{-1/3}\biggl|\int_{0}^{\tau+m-k}{f(y)}
\bar{\mathbf{\Psi}}_{(\nu-1)(\tau+m-k)}(y)\,dy\biggr|
\biggr]^p\\&
=2^{-p\nu/3}\sum_{m\in\mathbb{Z}}\bigl(1+|m|\bigr)^{-8p/3}\biggr)^{1/p}
\biggl(\sum_{n\ge 0}\Bigl(\int_{Q_{\nu n}}w\Bigr)^p\biggl[
\sum_{k=0}^{n} (k+1)^{-1/3}\biggl|\int_{0}^{n-k}{f(y)}
\bar{\mathbf{\Psi}}_{(\nu-1)(n-k)}(y)\,dy\biggr|
\biggr]^p\\&\lesssim2^{-p\nu/3}
\sum_{n\ge 0}\Bigl(\int_{Q_{\nu n}}w\Bigr)^p\biggl[
\sum_{m=0}^{n} (n-m+1)^{-1/3}\biggl|\int_{0}^{m}{f(y)}
\bar{\mathbf{\Psi}}_{(\nu-1)m}(y)\,dy\biggr|
\biggr]^p.\end{align*} 
On the strength of \cite[Theorem 1.8]{PSU} and in view of \eqref{Nd},
\begin{multline}\label{Proof2}
\biggl(\sum_{\tau\in\mathbb{Z}}\Bigl(\int_{Q_{\nu\tau}}w\Bigr)^p
\bigl|\langle I_{0^+}^{1/3}f,\widetilde{\mathbf{\Psi}}_{(\nu-1)\tau}\rangle\bigr|^p
\biggr)^{1/p}
\lesssim \mathscr{N}^{1/3}_{0^+}(\nu)2^{\nu/3}\biggl(\sum_{\tau\ge 0}\Bigl(\int_{Q_{\nu\tau}}u\Bigr)^p\biggl|\int_{0}^{\tau}{f(y)}\bar{\mathbf{\Psi}}_{(\nu-1)\tau}(y)\,dy\biggr|^p\biggr)^{1/p},
\end{multline} where $\bar{\mathbf{\Psi}}_{(\nu-1)m}$ in combination with ${\mathbf{\Phi}}_{2,\tau-3}$ forms the second order spline wavelet system, and 
$$\langle f,\bar{\mathbf{\Psi}}_{(\nu-1)\tau}\rangle/\mathbf{\Lambda}_2^{''}= 
\langle f(\cdot),{\mathbf{\Psi}}_{2,-3,-6}(2^{\nu-1}\cdot-\tau)\rangle/\mathbf{\Lambda}_2^{''}$$ with 
$\langle f(\cdot),{\mathbf{\Phi}}_{2,\tau-3}(\cdot)/\mathbf{\Lambda}'_2\rangle$ are the related to this system decomposing coefficients for $f$ in $B_{pq}^{1/3}(\mathbb{R},u)$. 
From here, on the strength of Theorem \ref{main'}, we approach the inequality \eqref{ineqEx}.\end{example}

For giving an idea how to perform a type of the reverse inequality for \eqref{left} we demonstrate the following
\begin{example}\label{Exam2}
Let $1<p,q<\infty$. Suppose $f\in L_1^\textrm{loc}(\mathbb{R})$ and $f(y)\equiv 0$ if $y\in(-\infty,0)$. Then
it holds 
\begin{equation}\label{ineqExR}\|f\|_{{B}_{pq}^{-1/3}(\mathbb{R})}\lesssim\|I_{0^+}^{1/3}f\|_{{B}_{pq}^{0}(\mathbb{R})}.\end{equation}

This time again $J=1$ for ${B}_{pq}^{-1/3}(\mathbb{R})$ and we choose $M=2$ and put $N=0$. By taking the fractional order scaling function 
$\beta_-^{4/3,0}=:\beta_-^{4/3}$ and the related wavelet $\mathit{\Psi}_-^{4/3}$ we write, making use of Theorem \ref{FrDec},
\begin{equation}\label{ryshaR}\|f\|_{B^{-1/3}_{pq}(\mathbb{R})}\lesssim\|{\lambda}\|_{{b}_{pq}^{-1/3}(w=1)}= 
\biggl\|\sum_{\nu\in\mathbb{N}_0}2^{-q\nu/3}\Bigl\|\sum_{\tau\in\mathbb{Z}}
|\lambda_{\nu\tau}|\chi_{Q_{\nu\tau}}\Bigr\|_{L^p(\mathbb{R})}^q\biggr\|^{1/q},\end{equation} where (see \eqref{vazhnoo} and \eqref{ForRepr'} with 
$\beta_-^{4/3}$ and $\mathit{\Psi}_-^{4/3}$)
$$\lambda_{0\tau}=\langle f,\widetilde{\mathbf{\Phi}}_{\tau}\rangle\quad(\tau\in\mathbb{Z}),\qquad   
\lambda_{\nu\tau}=2^{\nu/2}\langle f,\widetilde{\mathbf{\Psi}}_{\nu\tau}\rangle\quad(\nu\in\mathbb{N},\,\tau\in\mathbb{Z}).$$
As in the previous example, we start from $\nu=0$ in the right hand side of \eqref{ryshaR} and, by using the representation 
\begin{equation}\label{REPR}f(x)=\frac{1}{\Gamma(1-\alpha)\Gamma(\alpha)}\frac{d}{dx}
\int_{0}^x (x-y)^{-\alpha}\Bigl(\int_0^y f(t)(y-t)^{\alpha-1}dt\Bigr)dy\qquad(0<\alpha<1)\end{equation} (see \cite[\S\,2.3]{SKM}) with $\alpha=1/3$, 
we write for $\tau>0$:
\begin{align*}\frac{B(2/3,1/3)}{c_0}\langle f,\widetilde{\mathbf{\Phi}}_{\tau}\rangle=
\int_{\mathbb{R}}f(x) 
\beta_-^{4/3}(x-\tau)\,dx =\frac{B(2/3,1/3)}{\Gamma(7/3)}\sum_{k\ge 0}(-1)^k\binom{7/3}{k}\int_0^{\tau-k}f(x)
(-x+\tau-k)^{4/3}\,dx\\=\frac{4/3}{\Gamma(7/3)}\sum_{k\ge 0}(-1)^k\binom{7/3}{k}\int_0^{\tau-k}\biggl( \int_{0}^x (x-y)^{-1/3}\Bigl(\int_0^y f(t)(y-t)^{-2/3}dt\Bigr)dy\biggr) 
(-x+\tau-k)^{1/3}\,dx\\=\frac{4/3}{\Gamma(7/3)}\sum_{k\ge 0}(-1)^k\binom{7/3}{k}\int_0^{\tau-k}I_{0^+}^{1/3}f(y)\biggl(\int_y^{\tau-k}(-x+\tau-k)^{1/3}(x-y)^{-1/3}dx\biggr) 
\,dy.\end{align*} Since \begin{align}\label{1/3}\int_y^{\tau-k}(-x+\tau-k)^{1/3}(x-y)^{-1/3}dx=\frac{1}{3}\int_y^{\tau-k}\Bigl(\int_x^{\tau-k}(z-x)^{-2/3}dz\Bigr)(x-y)^{-1/3}dx\nonumber\\
=\frac{1}{3}\int_y^{\tau-k}\Bigl(\int_y^{z}(z-x)^{-2/3}(x-y)^{-1/3}dx\Bigr)dz=B(2/3,1/3)(-y+\tau-k),
\end{align} then
\begin{align*}{c_0}^{-1}\langle f,\widetilde{\mathbf{\Phi}}_{\tau}\rangle=
\frac{4/3}{\Gamma(7/3)}\sum_{k\ge 0}(-1)^k\binom{7/3}{k}\int_0^{\tau-k}I_{0^+}^{1/3}f(y)(-y+\tau-k) 
\,dy,\end{align*} which leads us, by manipulations analogous to those in Example \ref{Exam1}, to 
\begin{align*}{c_0}^{-1}\langle f,\widetilde{\mathbf{\Phi}}_{\tau}\rangle=
\frac{8/3}{\Gamma(7/3)}\sum_{k=0}^{\tau}(-1)^k\binom{1/3}{k}\int_0^\infty I_{0^+}^{1/3}f(y) B_1(-y+\tau-2-k)\,dy\\=:\frac{8/3}{\Gamma(7/3)}\sum_{k=0}^{\tau}(-1)^k\binom{1/3}{k}\langle I_{0^+}^{1/3}f,\mathbf{\Phi}_{1,\tau-2-k}\rangle.\end{align*} Further, by H\"{o}lders inequality and by virtue of \eqref{BinCoef},
\begin{align}\label{ppp}\sum_{\tau\ge 0}\bigl|\langle f,\widetilde{\mathbf{\Phi}}_{\tau}\rangle\bigr|^p\lesssim\sum_{\tau\ge 0}\biggl(\sum_{k=0}^\tau\Bigl|\binom{1/3}{k}\Bigr|\biggr)^{p-1}
\sum_{k=0}^{\tau}\Bigl|\binom{1/3}{k}\Bigr|\bigl|\langle I_{0^+}^{1/3}f,\mathbf{\Phi}_{1,\tau-2-k}
\rangle\bigr|^p\nonumber\\\lesssim
\sum_{l\ge 0}\Bigr|\bigl|\langle I_{0^+}^{1/3}f,\mathbf{\Phi}_{1,l-2}\rangle\bigr|^p\sum_{\tau\ge l}\Bigl|\binom{1/3}{\tau-l}\Bigr|\lesssim\sum_{l\ge 0}\Bigr|\bigl|\langle I_{0^+}^{1/3}f,\mathbf{\Phi}_{1,l-2}\rangle\bigr|^p.
\end{align} Observe that for an estimation of the same type in a weighted case one should have $r_w<\boldsymbol{\alpha}$. The point is that for jumping, by making use of \eqref{Muck}, from $\int_{Q_{0\tau}}w$ to $\int_{Q_{0(\tau-k)}}w$ it must be $\sum_{\tau\ge l}\Bigl|\binom{\boldsymbol{\alpha}}{\tau-l}\Bigr|(\tau-l)^{r_w}<\infty$.

For $\nu>0$ in the right hand side of \eqref{ryshaR} we obtain, by analogy to the case $\nu=0$, by using \eqref{ChVan},
\begin{align*}& 2^{4/3}{\Gamma(7/3)}(\langle f,\psi^{4/3}_-(2^{\nu-1}\cdot-\tau)\rangle\\&={\Gamma(7/3)}
\sum_{m\in\mathbb{Z}}(-1)^m\beta_\ast^{11/3}(m-1)\sum_{n\ge 0}(-1)^n\binom{7/3}{n}\int_{\mathbb{R}}f(x) 
\beta_-^{4/3}(2^\nu x-m+n-\tau)\,dx\\&=\!\!\sum_{m\in\mathbb{Z}}(-1)^m\beta_\ast^{11/3}(m-1)\!\sum_{n\ge 0}(-1)^n\binom{7/3}{n}
\!\sum_{l\ge 0}(-1)^l\binom{7/3}{l}\!\!\int_{\mathbb{R}}f(x) 
(-2^\nu x+m-n-l+\tau)_+^{4/3}\,dx\\&=\!\!\sum_{m\in\mathbb{Z}}(-1)^m\beta_\ast^{11/3}(m-1)\!\sum_{k\ge 0}(-1)^k\binom{14/3}{k}\int_{\mathbb{R}}f(x)(-2^\nu x+m-k+\tau)_+^{4/3}\,dx
,\end{align*}
where, in view of \eqref{REPR} and on the strength of \eqref{1/3} (see also \eqref{09_07_01}), 
\begin{align*}\int_{\mathbb{R}}f(x) 
(-2^\nu x+m-k+\tau)_+^{4/3}\,dx=&\int_0^{\frac{\tau+m-k}{2^\nu}}f(x) 
(-2^\nu x+m-k+\tau)^{4/3}\,dx\\=&4\cdot2^{\nu}/3\int_0^{\frac{\tau+m-k}{2^\nu}} I_{0^+}^{1/3}f(y)
(-2^\nu y+m-k+\tau)\,dy.\end{align*} 

To obtain a wavelet function $\mathbf{\Psi}_{1,-2,-4}$ of the first order related to $\mathbf{\Phi}_{1,-2}$ (see \eqref{ppp}), we need to form
$\sum_{n\ge 0}(-1)^n\binom{2}{n}B_1(2^\nu y-\tau-m+n+2)$, which can be written through iterated differences as
$$\Delta_{2^{-\nu}}^2B_1(2^\nu y-\tau-m+4)=\Delta_{2^{-\nu}}^4 (-2^\nu y+\tau+m-4)^1_+.$$ To do this, as before, we use \eqref{ChVan} with $s=4$ and $r=2/3$, and arrive to
$$\sum_{k\ge 0}(-1)^k\binom{14/3}{k}(-2^\nu x+\tau+m-k)_+=2\sum_{k=0}^{\tau+m}(-1)^{k+1}\binom{2/3}{k}
\Delta_{2^{-\nu}}^2B_1(2^\nu x-\tau-m+k+4).$$ Thus, by summarising the above assessments, and by making an estimate similar to \eqref{12_07_01} (with $w=1$ and $f$ 
instead of $I_{0^+}^{1/3}$ and vice versa), we come to
\begin{align*}
\biggl(\sum_{\tau\in\mathbb{Z}}\bigl|\langle f,\widetilde{\mathbf{\Psi}}_{(\nu-1)\tau}\rangle\bigr|^p
\biggr)^{1/p}
\lesssim & 2^{\nu/3}\biggl(\sum_{\tau\in\mathbb{Z}}\biggl[
\sum_{m\in\mathbb{Z}}\bigl|\beta_*^{11/3}(m-1)
\bigr|\sum_{k=0}^{\tau+m} \binom{2/3}{k}\\&\times\biggl|\int_{0}^{\frac{\tau+m-k}{2^\nu}}{I_{0^+}^{1/3}f(x)}\mathbf{\Psi}_{1,-2,-4}(2^{\nu-1} x-\tau-m+k)\,dy\biggr|
\biggr]^p\biggr)^{1/p}\end{align*}
with
$$\mathbf{\Psi}_{1,-2,-4}(2^{\nu-1} x-\tau-m+k):=\sum_{j=0}^1\frac{\lambda_j(1)}{2(-1)^j}
\Bigl[\Delta_{2^{-\nu}}^2B_1(2^\nu x-\tau-m+k+5+j)-\Delta_{2^{-\nu}}^2B_1(2^\nu x-\tau-m+k+5-j).$$ From here, similarly to the case $\nu=0$ 
(see also Example \ref{Exam1}), we obtain by H\"{o}lder's inequality,
\begin{align}\label{ppp'}\sum_{\tau\in\mathbb{Z}}\bigl|\langle f,\widetilde{\mathbf{\Psi}}_{\nu-1,\tau}\rangle\bigr|^p&\lesssim2^{\nu/3}
\sum_{\tau\in\mathbb{Z}}\biggl[\sum_{m\ge -\tau}\bigl|\beta_\ast^{11/3}(m-1)\bigr|\sum_{k=0}^{\tau+m}\Bigl|\binom{2/3}{k}\Bigr|
\bigl|\langle I_{0^+}^{1/3}f,\mathbf{\Psi}_{1,-2,\tau+m-4-k}
\rangle\bigr|\biggr]^p\nonumber\\&\lesssim
\sum_{l\ge 0}\Bigr|\bigl|\langle I_{0^+}^{1/3}f,\mathbf{\Psi}_{1,-2,l-4}\rangle\bigr|^p,
\end{align} and the required inequality \eqref{ineqExR} follows now by the decomposition theorems for  unweighted $B_{pq}^s(\mathbb{R})$ \cite[Theorems 2.46 and 2.49]{Tr5} (see also \cite[Proposition 5]{NU} or \cite[Proposition 4.1]{JMAA}). To make an estimate similar to \eqref{ppp'}, analogously to \eqref{ppp}, we need to have $r_w<2\boldsymbol{\alpha}$ (see the comment after \eqref{ppp}). 
 \end{example}
 
\begin{remark}\label{rem1} Observe that in Example \ref{Exam1} the number of steps for obtaining the required estimates can be reduced to those performed for $\nu\in\mathbb{N}$ only, 
simply by adding the case $\nu=0$ at that stage as well. The reason for this is that $\langle f,\Delta^{n+1}_1\mathbf{\Phi}_{n,\boldsymbol{k}}\rangle$ can 
be estimated from above by $\langle f,\mathbf{\Phi}_{n,\boldsymbol{k}-i}\rangle$ with $i=0,\ldots,n+1$.\end{remark}

Basing on the ideas from Examples \ref{Exam1} and \ref{Exam2}, we can state our main results for fractional $\boldsymbol{\alpha}>0$. 
\begin{theorem}\label{ImagesA} Let $1<p<\infty$, $0<q\le\infty$, $s\in\mathbb{R}$, weights $u,w\in\mathscr{A}_\infty$ and $f\in L_1^\textrm{loc}(\mathbb{R})$. 
For fractional $\boldsymbol{\alpha}>0$ and $a\in\mathbb{R}$ let $I_{a^\pm}^{\boldsymbol{\alpha}}$ be defined by \eqref{left} and \eqref{right}. 
Suppose $f\equiv 0$ on $(-\infty,a)$ or $(a,\infty)$, respectively. \\ {\rm (i)} Then
$I_{a^\pm}^{\boldsymbol{\alpha}}f\in {B}_{pq}^{s}(\mathbb{R},w)$ if $f\in {B}_{pq}^{s+\boldsymbol{\alpha}}(\mathbb{R},w)$ provided
$C_{a^\pm}^{\boldsymbol{\alpha}}:=\sup_{d\in\mathbb{N}_0}\mathscr{N}_{a^\pm}^{\boldsymbol{\alpha}}(d)<\infty$, where
\begin{align*}\mathscr{N}_{a^+}^{\boldsymbol{\alpha}}(d):=
&\frac{1}{2^{2d\boldsymbol{\alpha}}}\Biggl[\sup_{\tau\in\mathbb{Z}_a^+}\biggl(\sum_{r\ge\tau}(r-\tau+1)^{p(2\boldsymbol{\alpha}-1)}
\int_{Q_{dr}^{[a]}}w\biggr)^{{1}/{p}}\biggl(\sum_{[a]\le r\le\tau} \Bigl(\int_{Q_{dr}^{[a]}}\tilde{u}\Bigr)^{1-p'}\biggr)^{{1}/{p'}}\\&+
\sup_{\tau\in\mathbb{Z}_a^+}\biggl(\sum_{r\ge\tau}\int_{Q_{dr}^{[a]}}w\biggr)^{{1}/{p}}\biggl(\sum_{[a]\le r\le\tau} (\tau-r+1)^{p'(2\boldsymbol{\alpha}-1)}
\Bigl(\int_{Q_{dr}^{[a]}}\tilde{u}\Bigr)^{1-p'}\biggr)^{{1}/{p'}}\Biggr],\nonumber\\
\mathscr{N}_{a^-}^{\boldsymbol{\alpha}}(d):=
&\frac{1}{2^{2d\boldsymbol{\alpha}}}\Biggl[\sup_{\tau\in\mathbb{Z}_a^-}\biggl(\sum_{[a]\ge r\ge\tau}(r-\tau+1)^{p(2\boldsymbol{\alpha}-1)}
\int_{Q_{dr}^{[a]}}w\biggr)^{{1}/{p}}\biggl(\sum_{r\le\tau} \Bigl(\int_{Q_{dr}^{[a]}}\tilde{u}\Bigr)^{1-p'}\biggr)^{{1}/{p'}}\\&+
\sup_{\tau\in\mathbb{Z}_a^-}\biggl(\sum_{[a]\ge r\ge\tau}\int_{Q_{dr}^{[a]}}w\biggr)^{{1}/{p}}\biggl(\sum_{r\le\tau} (\tau-r+1)^{p'(2\boldsymbol{\alpha}-1)}
\Bigl(\int_{Q_{dr}^{[a]}}\tilde{u}\Bigr)^{1-p'}\biggr)^{{1}/{p'}}\Biggr]
\end{align*} with $\mathbb{Z}^+_a:=\mathbb{Z}\cap\bigl[[a],\infty\bigr)$ and $\mathbb{Z}^-_a:=\mathbb{Z}\cap\bigl(-\infty,[a]\bigr]$, 
${Q_{dr}^{[a]}}:=\Bigl[{\frac{r-[a]}{2^d}},{\frac{r-[a]+1}{2^d}}\Bigr]$,
and $\tilde{u}\le u$ on $\mathbb{R}$. Moreover, 
\begin{equation}\label{more+}\|I_{a^\pm}^{\boldsymbol{\alpha}}f\|_{{B}_{pq}^{s}(\mathbb{R},w)}\lesssim C_{a^\pm}^{\boldsymbol{\alpha}}
\|f\|_{{B}_{pq}^{s+\boldsymbol{\alpha}}(\mathbb{R},u)}.\end{equation}
{\rm (ii)} If $I_{a^\pm}^{\boldsymbol{\alpha}}f\in {A}_{pq}^{s}(\mathbb{R},w)$ then $f\in {B}_{pq}^{s-\boldsymbol{\alpha}}(\mathbb{R},w)$ 
provided $r_w<\boldsymbol{\alpha}$, besides,
\begin{equation}\label{more-}\|f\|_{{B}_{pq}^{s-\boldsymbol{\alpha}}(\mathbb{R},w)}\lesssim\|I_{a^\pm}^{\boldsymbol{\alpha}}f\|_{{B}_{pq}^{s}(\mathbb{R},w)}.
\end{equation} 
For $\boldsymbol{\alpha}\in(0,1)$ the assertion {\rm(ii)} of the theorem is unconditionally true in the case $w\equiv 1$.
\end{theorem}

\begin{proof} (i) We need to introduce two spline wavelet systems with orders suitable for decomposing norms in the both sides of \eqref{more+}.
To this end we determine $J$ for $B_{pq}^{s}(\mathbb{R},w)$ and define $\delta$, $M>0$ and $N\ge -1$ according to (M.i) -- (M.iii). 
Therefore (see \S\,\ref{FR}), for $B_{pq}^{s}(\mathbb{R},w)$ one can choose $\alpha_0\ge \begin{cases} M-1+[s], & s\ge -1,\\ M-2-s, & s<-1.\end{cases}$
Besides, $\alpha^\ast\in\mathbb{N}$ for $B_{pq}^{s+\boldsymbol{\alpha}}(\mathbb{R},u)$ must satisfy \eqref{condB}.
Therefore, by following the idea from Example \ref{Exam1}, $\alpha_0$ and $\alpha^\ast$ must be taken in such a way to comply the condition $\alpha_0+\boldsymbol{\alpha}=\alpha^\ast$. Besides,  for ability to perform an estimate of the type \eqref{12_07_02} in the proof, one should fix $\alpha_0$ enough big to have $r_w<2(\alpha_0+1)$. 

Consider the operator $I_{a^+}^{\boldsymbol{\alpha}}$. On the strength of Theorem \ref{FrDec},
\begin{equation}\label{rysha+}\|I_{a^+}^{\boldsymbol{\alpha}}f\|_{B^{s}_{pq}(\mathbb{R},w)}\lesssim\|{\lambda}\|_{{b}_{pq}^{s}(w)}= 
\biggl\|\sum_{\nu\in\mathbb{N}_0}2^{q\nu s}\Bigl\|\sum_{\tau\in\mathbb{Z}}
|\lambda_{\nu\tau}|\chi_{Q_{\nu\tau}^{[a]}}\Bigr\|_{L^p(\mathbb{R},w)}^q\biggr\|^{1/q},\end{equation} where (see \eqref{vazhnoo} and \eqref{ForRepr'} with 
$\beta_-^{\alpha_0,\boldsymbol{k}}$ and $\mathit{\Psi}_-^{\alpha_0,\boldsymbol{k},\boldsymbol{s}}$)
$$\lambda_{0\tau}=\langle I_{a^+}^{\boldsymbol{\alpha}}f,\widetilde{\mathbf{\Phi}}_{\tau}\rangle\quad(\tau\in\mathbb{Z}),\qquad   
\lambda_{\nu\tau}=2^{\nu/2}\langle I_{a^+}^{\boldsymbol{\alpha}}f,\widetilde{\mathbf{\Psi}}_{\nu\tau}\rangle\quad(\nu\in\mathbb{N},\,\tau\in\mathbb{Z}).$$ 
One can fix $\boldsymbol{k}\le[a]$ and $\boldsymbol{s}\le[a]$, in order to have $\langle I_{a^+}^{\boldsymbol{\alpha}}f,
\widetilde{\mathbf{\Phi}}_{\tau}\rangle=0$ for $\tau\le 0$.

Further considerations are similar to those in Example \ref{Exam1} (see also Remark \ref{rem1}). Starting from the right hand side of \eqref{rysha+} 
one should estimate it from above by  
$\|\lambda^\ast\|_{b_{pq}^{s+\boldsymbol{\alpha},u}}$ with \begin{equation*}\label{seq0'}\lambda^\ast_{0\tau}=\langle f,\widetilde{\mathbf{\Phi}}^\ast_{\tau}
\rangle\quad(\tau\in\mathbb{Z}),\qquad   \lambda^\ast_{\nu\tau}=2^{\nu}\langle f,\widetilde{\mathbf{\Psi}}^\ast_{(\nu-1)\tau}\rangle
\quad(\nu\in\mathbb{N},\,\tau\in\mathbb{Z}),\end{equation*} where
$$\widetilde{\mathbf{\Phi}}^\ast_{\tau}=\bigl(\mathbf{\Lambda}'_{\alpha^\ast}\bigr)^{-1}{\mathbf{\Phi}}_{\alpha^\ast,
\boldsymbol{k}^\ast=-\boldsymbol{k}-\alpha^\ast-1}(\cdot-\tau),$$ and $$
\widetilde{\mathbf{\Psi}}_{(\nu-1)\tau}^\ast=\bigl(\mathbf{\Lambda}^{''}_{\alpha^\ast}\bigr)^{-1}2^{(\nu-1)}{\mathbf{\Psi}}_{\alpha^\ast,
\boldsymbol{k}^\ast =-\boldsymbol{k}-\alpha^\ast-1,\boldsymbol{s}^\ast=-\boldsymbol{s}-2(\alpha^\ast+1)}(2^{\nu-1}\cdot-\tau),
$$ by performing for $\nu\in\mathbb{N}_0$ all the steps starting from \eqref{Proof1} and finishing by \eqref{Proof2}, with chosen $\alpha_0$ and $\alpha^\ast$. 

From this (i) follows by applying Theorem \ref{main'} with $\alpha^\ast$. For proving the validity of \eqref{more+} with $I_{a^-}^{\boldsymbol{\alpha}}$ one
should apply to Theorem \ref{FrDec} a fractional order spline wavelet system of the type $\{\beta_+^\alpha,\mathit{\Psi}_+^\alpha\}$.

(ii) To prove \eqref{more-} assume $I_{a^\pm}^{\boldsymbol{\alpha}}f\in {B}_{pq}^{s}(\mathbb{R},w)$ and fix some natural  
$n_0$ fitting the condition \eqref{condB} with respect to $w$. Besides, we determine $J$ for $B_{pq}^{s-\boldsymbol{\alpha}}(\mathbb{R},w)$ 
and define $\delta$, $M>0$ and $N\ge -1$ according to (M.i) -- (M.iii). Further, we choose 
$\alpha_\ast\ge \begin{cases} M-1+[s-\boldsymbol{\alpha}], & s\ge \boldsymbol{\alpha}-1,\\ M-2-s-\boldsymbol{\alpha}, & s<\boldsymbol{\alpha}-1.\end{cases}$
Accordingly to the idea from Example \ref{Exam2}, the order $\alpha_\ast$ should satisfy the condition 
$\alpha_\ast-\boldsymbol{\alpha}=\alpha_0\ge n_0$, where $\alpha_0$ is natural.

For proving \eqref{more-} with $I_{a^+}^{\boldsymbol{\alpha}}$, one can start from applying Theorem \ref{FrDec}, which implies the estimate
\begin{equation*}\|f\|_{B^{s-\boldsymbol{\alpha}}_{pq}(\mathbb{R},w)}\lesssim\|{\lambda_\ast}\|_{{b}_{pq}^{s-\boldsymbol{\alpha}}(w)}= 
\biggl\|\sum_{\nu\in\mathbb{N}_0}2^{q\nu(s-\boldsymbol{\alpha})}\Bigl\|\sum_{\tau\in\mathbb{Z}}
|\lambda_{\ast\nu\tau}|\chi_{Q_{\nu\tau}^{[a]}}\Bigr\|_{L^p(\mathbb{R},w)}^q\biggr\|^{1/q},\end{equation*} where elements \eqref{vazhnoo} and \eqref{ForRepr'} 
are defined with $\beta_-^{\alpha_\ast,\boldsymbol{k}_\ast}$ and $\mathit{\Psi}_-^{\alpha_\ast,\boldsymbol{k}_\ast,\boldsymbol{s}_\ast}$ 
(and proper $\boldsymbol{k}_\ast,\boldsymbol{s}_\ast$), that is
$$\lambda_{\ast 0\tau}=\langle f,\widetilde{\mathbf{\Phi}}_{\tau}\rangle\quad(\tau\in\mathbb{Z}),\qquad   
\lambda_{\ast\nu\tau}=2^{\nu/2}\langle f,\widetilde{\mathbf{\Psi}}_{\nu\tau}\rangle\quad(\nu\in\mathbb{N},\,\tau\in\mathbb{Z}).$$ 
Further, we follow the idea from Example \ref{Exam2} to approach (from above) the norm on the right hand side of \eqref{rysha+}. 
This could be achieved analogously to the method described in Example \ref{Exam2} in the case $w\equiv 1$, additionally supplied with an estimate of the type
\eqref{12_07_02} if $w\not\equiv 1$. The rest follows by Theorem \ref{main'}.
\end{proof}

\begin{corollary} It follows from Theorem \ref{ImagesA} that under the condition \eqref{usl} for the both weights $u$ and $w$ it holds 
$C_{a^\pm}^{\boldsymbol{\alpha}}=\mathscr{N}_{a^\pm}^{\boldsymbol{\alpha}}(0)$ {\rm(}see Theorem \ref{T11} for the case $a=0$ in the Introduction{\rm)}.
\end{corollary} 

\begin{remark} The case $0<p\le 1$ can be also involved into consideration in Theorem \ref{ImagesA} with properly modified conditions $C_{a^\pm}^{\boldsymbol{\alpha}}<\infty$ (see \cite[\S\,1.4]{PSU} and \cite[Chapter 11, Section 1.5, Theorem 4]{KA}).
\end{remark}

\end{document}